\documentclass{article}

\usepackage{placeins} 
\usepackage{hyperref}       
\usepackage{xcolor}  
\usepackage[numbers]{natbib}
\usepackage[utf8]{inputenc} 
\usepackage{chngcntr}
\usepackage{amsthm}
\usepackage[T1]{fontenc}    
\usepackage{hyperref}       
\usepackage{url}            
\usepackage{booktabs}       
\usepackage{amsfonts}       
\usepackage{nicefrac}       
\usepackage{microtype}      
\usepackage{xcolor}         
\usepackage{subcaption} 

\usepackage{algorithm}
\usepackage{algpseudocode}
\usepackage{amsfonts}
\usepackage{amsmath}

\usepackage{amssymb}
\usepackage{comment}
\usepackage{cleveref}  

\usepackage{eurosym}
\usepackage{graphicx}
\usepackage{epsfig}
\usepackage{hyperref}
\usepackage{dsfont}
\usepackage{color}
\usepackage{xcolor}
\usepackage{appendix}

\newtheorem{theorem}{Theorem}[section]

\newtheorem{definition}[theorem]{Definition}	
\newtheorem{proposition}[theorem]{Proposition}
\newtheorem{corollary}[theorem]{Corollary}	
	
\newtheorem{problem}[theorem]{Problem}

\newtheorem{assumption}[theorem]{Assumption}

\counterwithin{theorem}{section}

\newcommand{\Rd}{\mathbb{R}^d}
\newcommand{\R}{\mathbb{R}}

\usepackage{geometry}
\title{Flow Matching for Measure Transport and Feedback Control of Control-Affine Systems}

\author{%
  Karthik Elamvazhuthi\thanks{The author is affiliated with Los Alamos National Laboratory. Email-id: \texttt{karthikevaz@lanl.gov}}
}

\begin{document}

\maketitle

\begin{abstract}

We develop a \emph{flow-matching framework} for transporting probability measures under control-affine dynamics and for steering systems to points or target sets. Starting from the continuity equation associated with the control affine system
\[
\dot x = f_0(x) + \sum_{i=1}^m u_i f_i(x),
\]
we construct measure interpolations through \emph{exact} and \emph{approximate flow matching}, and extend the approach to \emph{output flow matching} when only output distributions must align. These constructions allow to directly import standard control tools, such as feedback design, oscillatory inputs, and trajectory steering, and yield sample-efficient, regression-based feedback controllers for measure-to-measure transport. 

We also introduce a complementary \emph{``noising + time-reversal'' perspective} for classical state or set stabilization, inspired by denoising diffusion models. Here stabilization is interpreted as a \emph{denoising problem}: noising corresponds to destabilizing the system through excitations, while denoising corresponds to stabilization via time reversal. We propose two methods for constructing the noising process:  

(i) \emph{Randomized-control noising}, which employs regular (non-white noise) controls through the endpoint map and naturally accommodates control constraints.  

(iI) \emph{PMP-based noising}, which leverages the Hamiltonian system from Pontryagin's Maximum Principle, corresponding the fixed or variable end-point optimal control problems, to explore the configuration space by randomizing the adjoint vectors and recovers the optimal controller for linear systems with convex costs, while providing feasible feedback laws in the nonlinear case.  

We establish existence of solutions to the corresponding ODEs and regularity of the induced flows on measures, even when control laws are nonsmooth.

Finally, we numerically illustrate the framework on linear and nonlinear systems, demonstrating its effectiveness for measure transport, steering systems to target sets and path planning in a domain with obstacles.

\end{abstract}

\section{Introduction}

Sampling from an unknown distribution is a fundamental problem in machine learning and statistics. Towards this end, a number of methods have been introduced in the literature. Examples include normalizing flows,  optimal transport \cite{onken2021ot}, time reversal of diffusion processes \cite{song2020score}, and flow matching\cite{lipmanflow,liu2022flow,albergo2023stochastic} Among these, diffusion models and flow matching have become particularly attractive: they bypass constrained optimization problems inherent in optimal transport formulations, and instead reduce the problem to solving a regression task, which leads to scalable algorithms for fast sampling.

Motivated by these developments in generative modeling, the present paper investigates the potential of flow matching for problems arising in control theory. Specifically, we consider control-affine systems of the form
\begin{equation}
\dot{x}(t) = f(x,u) = f_0(x) + \sum_{i=1}^m u_i(t) f_i(x),
\label{eq:ctrsys}
\end{equation}
where $f_i:\mathbb{R}^d \to \mathbb{R}^d$ are smooth vector fields for $i=0,\dots,m$, and the case $m<d$ is of particular interest. Our objective is to shape the probability distribution of the state process $x(t)$. If $\mu_t$ denotes the distribution of $x(t)$, its evolution is governed by the continuity equation
\begin{equation}
\partial_t \mu + \nabla \cdot \left( f(x,u(t,x)) \mu \right) = 0,
\quad (x,t)\in \mathbb{R}^d \times I,
\label{eq:ctnteq}
\end{equation}
where $I=[0,T]$.

This motivates the following problems.

\begin{problem}[Measure-to-measure control]
\label{prb1}
Given $\mu_0, \mu_T \in \mathcal{P}(\mathbb{R}^d)$, find a feedback law $u(t,x)$ such that the solution of \eqref{eq:ctnteq} satisfies $\mu_T$ at time $T$.
\end{problem}

A relaxation of the problem when instead of controlling to $\mu_T$ exactly, we only want to control the support of the target measure.

\begin{problem}[Steering to a target set]
\label{prb2}
Given $\mu_0\in\mathcal{P}(\mathbb{R}^d)$ and target set $\Omega \subset \mathbb{R}^d$, find $u(t,x)$ such that \eqref{eq:ctnteq} holds with terminal measure support of $\mu_T$ lies in $\Omega$.
\end{problem}

The goal of this paper is to adapt ideas from flow matching in generative modeling \cite{lipmanflow,liu2022flow,albergo2023stochastic} to address Problems~\ref{prb1} and \ref{prb2}. One established approach to these problems is via \emph{optimal transport} \cite{villani2008optimal,santambrogio2015optimal}. Classical optimal transport theory assumes that all paths are admissible, but recent work has extended the theory to settings where the transport cost arises from an optimal control problem \cite{agrachev2009optimal,figalli2010mass,hindawi2011mass,chen2016optimal,elamvazhuthi2025optimal,elamvazhuthi2024benamou,einav2025equivalence,citti2025benamou}. In particular, \cite{elamvazhuthi2024benamou} established equivalence between the Benamou–Brenier and Kantorovich formulations of optimal transport, inspired by earlier work on averaging random vector fields through Young measures \cite{lisini2006absolutely,bernard2008young}. This perspective provides new insights into controllability properties of the continuity equation \cite{arguillere2017sub,duprez2019approximate}, though it may introduce issues of non-uniqueness.

Computational approaches to transport for control systems have also been studied extensively \cite{elamvazhuthi2018optimal,elamvazhuthi2023dynamical,ito2023entropic,wu2024duality,craig2025blob,elamvazhuthi2025optimal}. However, existing methods often face difficulties in high dimensions due to state-space gridding \cite{elamvazhuthi2018optimal,elamvazhuthi2023dynamical} or sampling complexity \cite{craig2025blob}. This motivates exploring whether generative modeling techniques, which have demonstrated remarkable scalability, can be adapted to transport problems in control.

Recent works have applied diffusion-based generative models to control problems. For example, \cite{grong2024score} extends score-based diffusion methods to a stochastic version of Problem~\ref{prb2} in the special case $f \equiv 0$. In \cite{elamvazhuthi2025score}, score-based diffusion is developed for more general initial and terminal measures in driftless control-affine and linear systems. The work \cite{mei2025time} further studies steering to Dirac measures for stochastic systems with drift. 

These developments naturally raise the question of whether \emph{flow matching} \cite{lipmanflow,liu2022flow,albergo2023stochastic} can also be adapted to control. Flow matching has implicitly been employed in the study of optimal transport with control costs, where it was used to establish feasibility results in \cite{elamvazhuthi2024benamou}, and more recently its applicability was demonstrated through numerical experiments on stochastic linear systems in \cite{mei2025flow}. These works highlight the potential of flow matching as a computational alternative to existing methods for measure interpolation, but they remain either primarily theoretical or restricted to special cases.

A related but distinct line of work from some of the ideas presented in this manuscript is \emph{adjoint flow matching} \cite{domingo2024adjoint}, which employs adjoint dynamics to guide a learned flow model post training. By contrast, our use of adjoint equations is parameterizing \emph{noising processes} whose time reversal yields feedback laws. This viewpoint enables us to address a broader family of control-affine systems and cost functions, and connects flow matching with classical tools from the Pontryagin maximum principle for the synthesis of feedback controls, without resorting to the solution of a Hamilton–Jacobi–Bellman equation.

Finally, we note that the perspective of lifting classical control problems to the space of measures, in the sense of Problem \ref{prb2}, is a known idea in control theory,   \cite{borkar2002convex,hernandez1996linear,lasserre2008nonlinear,majumdar2014convex}. These works motivate the contributions of the latter portions of the paper.

\section*{Contributions}

The main contributions of this paper are as follows:

\begin{enumerate}
    \item \textbf{Flow matching for control-affine systems.} 
    We extend flow matching, originally developed in generative modeling, to control-affine dynamics. By averaging along admissible trajectory–control pairs, we obtain feedback laws that solve the continuity equation exactly, and introduce approximate and output-based variants useful when exact steering or full-state matching is not feasible.

    \item \textbf{Feedback based steering to target sets via noising and time reversal.} 
    We propose a regular version of the viewpoint \cite{elamvazhuthi2025score,mei2025time} on state stabilization as constructing \emph{noising processes} whose time reversal yields feedback controls that steer all initial states to a target state or set. Two parameterizations are developed:
    \begin{enumerate}
        \item a \emph{randomized-control noising} scheme that uses regular (non-white noise) controls. This also potentially provides a model free approach to perform flow matching based control if one has access to a simulator of the time-reversed system.
        \item a \emph{PMP-based noising} scheme that leverages adjoint trajectories (exponential map in the driftless, minimum-energy case) to noise the system by propagating sampled trajectories onto the reachable set.    
    \end{enumerate}

    \item \textbf{Analytical properties of constructed measure interpolations} 
    We establish existence and regularity of the induced flows through optimal transport theory and properties of the exponential map \cite{agrachev2009any,agrachev2009optimal,rifford2005morse,rifford2009stabilization}, analyze support properties of measures relevant to stabilization, and connect our constructions with classical tools such as the Pontryagin maximum principle under conditions of  controllability and non-existence of the so called {\it abnormal extremals}.
\end{enumerate}

In addition to these contributions,  we also demonstrate the proposed methods on linear and nonlinear examples, illustrating measure-to-measure transport and stabilization.

\section{Flow Matching: From Generative Models to Control Systems}

Before addressing Problem~\ref{prb1}, we briefly recall the idea of \emph{flow matching} 
\cite{liu2022flow,lipmanflow,albergo2023stochastic}, which provides the main inspiration for this work. 
The presentation here follows the rectified flows framework of \cite{liu2022flow}, although we adopt the 
now-standard terminology of flow matching due to \cite{lipmanflow}.

\subsection{Flow matching in the unconstrained setting}

Suppose the goal is to construct a stochastic process $(X_t)_{t\in I}$ interpolating between two random 
variables $X_0 \sim \mu_0$ and $X_T \sim \mu_T$. A simple pathwise interpolation is given by
\[
S_t(x_0,x_T) = \frac{T-t}{T}x_0 + \frac{t}{T}x_T, \qquad t \in I,
\]
with the associated process
\[
X_t = S_t(X_0,X_T).
\]
Although $(X_t)$ connects the prescribed marginals, it is not, in general, the solution of a differential 
equation. To obtain a dynamical description, define the \emph{velocity field}
\[
v(t,x) = \mathbb{E}\big[\dot{X}_t \,\big|\, X_t = x\big],
\]
and consider the process $(Z_t)$ defined by the ODE
\[
\dot{Z}_t = v(t,Z_t).
\]
As argued in \cite{liu2022flow,lipmanflow,albergo2023stochastic}, the law of $Z_t$ coincides with that of 
$X_t$ for all $t\in I$.  This insight, is based on the known idea of {\it purification} of measure valued coefficients or controls of dynamical system  into deterministic functions by averaging \cite{balder1998lectures,gyongy1986mimicking,lisini2006absolutely,bernard2008young}. 

Thus, one can sample from $\mu_T$ by simulating this deterministic flow once 
$v$ is learned.

Algorithmically, the procedure is:

\begin{enumerate}
    \item Sample pairs $(X_0,X_T)$ from the joint distribution $\mu_0 \otimes \mu_T$ and form the 
    interpolation $X_t=S_t(X_0,X_T)$.
    \item Learn $v$ by solving the regression problem
    \[
    \min_{v} \int_0^T \mathbb{E}\!\left[ \| \dot{X}_t - v(t,X_t) \|^2 \right] dt.
    \]
    \item Generate new samples: set $Z_0 \sim \mu_0$ and integrate $\dot{Z}_t=v(t,Z_t)$ forward to obtain 
    $Z_T \sim \mu_T$.
\end{enumerate}

\subsection{Flow matching under control constraints}

In the control setting, interpolating paths cannot be chosen arbitrarily and they must satisfy the 
dynamics~\eqref{eq:ctrsys}. Following \cite{elamvazhuthi2024benamou}, for $I : = [0,T]$, we introduce an interpolation 
map 
\[
S : \mathbb{R}^d \times \mathbb{R}^d \to C(I;\mathbb{R}^d) \times L^2(I;\mathbb{R}^m),
\]
such that for each pair $(x_0,x_T)$, the trajectory $S^\omega_t(x_0,x_T)$ and control 
$S^u_t(x_0,x_T)$ satisfy
\[
\dot{S}^\omega_t(x_0,x_T) = f_0(S^\omega_t(x_0,x_T)) + \sum_{i=1}^m 
S^{u_i}_t(x_0,x_T) f_i(S^\omega_t(x_0,x_T)).
\]
Let $(X_t,U_t)$ denote the random process obtained by sampling $(X_0,X_T) \sim \gamma$, where 
$\gamma \in \mathcal{P}(\mathbb{R}^{2d})$ is a transport plan, and setting
\[
(X_t,U_t) = \big(S^\omega_t(X_0,X_T),\, S^u_t(X_0,X_T)\big).
\]

We define the feedback control law
\[
u(t,x) = \mathbb{E}[\,U_t \mid X_t = x\,],
\]
and consider the deterministic flow
\begin{equation}
\dot{Z}_t = f_0(Z_t) + \sum_{i=1}^m u_i(t,Z_t) f_i(Z_t).
\label{eq:finalflow}
\end{equation}
We will generalize this construction in coming sections and also prove that the following sections that $Z_t$ and $X_t$ have the same law, hence both processes 
realize the interpolation $\mu_t$. We summarize the generalization of flow matching in Algorithm \ref{alg:flowmatching}, which adapts the flow matching framework to control-affine dynamics. For numerical experiments verifying the efficacy of the algorithm we refer the reader to Section \ref{sec:numexp}. 



\begin{algorithm}
\caption{Flow Matching for Measure Interpolation of Control Systems}
\label{alg:flowmatching}
\begin{algorithmic}[1]
\State \textbf{Input:} initial and target measures $(\mu_0,\mu_T)$.
\State Sample pairs $(X_0,X_T) \sim \gamma$, where $\gamma$ is a coupling of $(\mu_0,\mu_T)$.
\State Define interpolation $(X_t,U_t)$ using the system-consistent map $S(x_0,x_T)$.
\State Learn feedback law $u(t,x)$ by regression:
\[
\min_{u} \int_0^T \mathbb{E}\!\left[ \| U_t - u(t,X_t) \|^2 \right] dt.
\]
\State Generate new samples: set $Z_0 \sim \mu_0$ and simulate
\[
\dot{Z}_t = f_0(Z_t) + \sum_{i=1}^m u_i(t,Z_t) f_i(Z_t).
\]
\State \textbf{Output:} terminal states $Z_T$ distributed according to $\mu_T$.
\end{algorithmic}
\end{algorithm}

\section{Notation and Preliminaries}
The constructions in the previous section rely on interpreting 
flow matching within the framework of control-affine systems. 
In order to make these connections precise, introduce generalizations and to prepare for the 
analysis that follows, we now introduce the notation, assumptions, 
and basic concepts used throughout the remainder of the paper. 
Unless otherwise stated, the following conventions will be adopted.

We work in the $d$-dimensional Euclidean space $\mathbb{R}^d$
with control space $\mathbb{R}^m$. Throughout this paper,
$I$ denotes the time interval with fixed terminal time $T > 0$.
We use $\langle \cdot, \cdot \rangle$ for the standard inner
product in Euclidean space, $|\cdot|$ for the Euclidean norm,
and $\|\cdot\|$ for other norms as specified by context. The Borel
$\sigma$-algebra on a topological space $X$ is denoted by
$\mathcal{B}(X)$.

For function spaces, we denote by $C(I; \mathbb{R}^d)$ the space
of continuous functions from $I$ to $\mathbb{R}^d$, equipped with
the supremum norm. The space $C^k(\mathbb{R}^d; \mathbb{R}^m)$
consists of $k$-times continuously differentiable functions,
while $C_b(X)$ denotes bounded continuous functions on $X$. For
measurable functions, $L^p(I; \mathbb{R}^m)$ represents the
Lebesgue space of $p$-integrable functions, and
$H^1(I; \mathbb{R}^d)$ denotes the Sobolev space of absolutely
continuous functions with square-integrable derivatives.

We denote by $\mathcal{P}(X)$ the space of Borel probability
measures on $X$, and by $\mathcal{P}_p(X)$ the subspace of
measures with finite $p$-th moment. For measures $\mu$ and $\nu$,
we write $\mu \ll \nu$ when $\mu$ is absolutely continuous with
respect to $\nu$. The pushforward of a measure
$\mu \in \mathcal{P}(X)$ under a measurable map $f: X \to Y$ is
denoted by $(f)_\#\mu \in \mathcal{P}(Y)$, defined by
$(f)_\#\mu(B) = \mu(f^{-1}(B))$ for Borel sets $B \subseteq Y$.
Product measures are denoted by $\otimes$, and $W_1(\mu,\nu)$
represents the $1$-Wasserstein distance between measures $\mu$ and
$\nu$.

In the context of control systems, we use $x(t) \in \mathbb{R}^d$
to denote deterministic state trajectories and
$u(t) \in U \subseteq \mathbb{R}^m$ for control inputs, where $U$
is the control constraint set. The dynamics are governed by
control-affine vector fields of the form
$f(x,u) = f_0(x) + \sum_{i=1}^m u_i f_i(x)$, where
$f_i: \mathbb{R}^d \to \mathbb{R}^d$ are vector fields for
$i = 0,1,\ldots,m$.

\begin{assumption}\label{asmp:sublin}
The drift and control vector fields satisfy the following
assumptions:
\begin{enumerate}
\item The vector fields $f_i$ are
$C^2(\mathbb{R}^d;\mathbb{R}^d)$ for each $i = 0,\ldots,m$. 
\item The vector fields $f_i$ have sublinear growth for each
$i = 0,\ldots,m$. That is, there exists $M>0$ such that 
\[|f_i(x)| \leq M(|x|+1)\]
for all $x \in \mathbb{R}^d$ and all $i = 0,\ldots,m$.
\end{enumerate}
\end{assumption}

We define projection maps
$\pi^t: I \times \mathbb{R}^d \times U \rightarrow I$,
$\pi^x: I \times \mathbb{R}^d \times U \rightarrow \mathbb{R}^d$
and $\pi^u: I \times \mathbb{R}^d \times U \rightarrow U$ by
\begin{align}
    \pi^t(t,x,u) = t, \quad
    \pi^x(t,x,u) = x, \quad
    \pi^u(t,x,u) = u \nonumber
\end{align}
for all $(t,x,u) \in I \times \mathbb{R}^d \times U$. Given any
$K \in \mathcal{P}(I \times X)$ with marginal
$\pi^t_{\#}K = \mathrm{leb}$ (Lebesgue measure on $I$), there
exists a corresponding disintegration $K_t$ such that 
\[ \int_{I \times X} f(t,x)dK(t,x) =
   \int_{I}\int_{X} f(t,x) dK_t(x)dt \]
for all functions $f \in C_b(I \times X)$. When we write $K_t$,
we mean the disintegration of $K$ with respect to the time
variable evaluated at $t$.

Let $\Gamma = C(I;\mathbb{R}^d)$ denote the space of continuous
trajectories. We define the set of admissible controls as
\[
\mathcal{U} = \{ u \in L^p(I;\mathbb{R}^m) \,;\,
u(t) \in U \text{ for a.e. } t \in I \}
\]
and the set of admissible trajectory–control pairs as
\[
\Omega := \{(\omega, u) \in \Gamma \times \mathcal{U} \;:\;
\dot{\omega}(t) = f(\omega(t), u(t)) \text{ for a.e. } t \in I \}.
\]

Given a map $S: M \rightarrow \Omega$ from a parameter space $M$
to the set of admissible pairs, we denote by $S^\omega_t(\cdot)$
and $S^u_t(\cdot)$ the trajectory and control components evaluated
at time $t$.

For stochastic elements, capital letters such as $X_t$ denote
stochastic processes or random variables. The expectation operator
is denoted by $\mathbb{E}[\cdot]$, with $\mathbb{E}[\cdot | \cdot]$
representing conditional expectation. We write $\mathbb{P}(A)$
for the probability of event $A$, and $X \sim \mu$ to indicate
that the random variable $X$ is distributed according to measure
$\mu$.

Regarding derivatives, we use $\partial_t = \frac{\partial}{\partial t}$
for partial derivatives with respect to time, and
$\nabla = \nabla_x$ for the gradient with respect to spatial
variables. The divergence operator is denoted by $\nabla \cdot$,
while $D_x f$ represents the Jacobian matrix of $f$ with respect
to $x$. For vector fields $f$ and $g$, their Lie bracket is
denoted by $[f,g]$.

A central object of study is the continuity equation for a
time-dependent probability measure
$I \ni t \mapsto \mu_t \in \mathcal{P}(\mathbb{R}^d)$ with vector field
$v: I \times \mathbb{R}^d \to \mathbb{R}^d$:
\begin{equation}
    \partial_t \mu_t + \nabla \cdot (v(t,x) \mu_t) = 0
    \quad \text{in } \mathbb{R}^d \times I,
\end{equation}
understood in the weak sense. That is, for all test functions
$\varphi \in C^\infty_c(\mathbb{R}^d \times I)$, we require
\begin{equation}
    \int_I \int_{\mathbb{R}^d} \left[
    \partial_t \varphi(t,x) +
    \langle \nabla_x \varphi(t,x), v(t,x) \rangle \right]
    \mu_t(dx) \, dt = 0.
\end{equation}
Accordingly, for control systems, the continuity equation for a
time-dependent probability measure
$\mu_t \in \mathcal{P}(\mathbb{R}^d)$ with control-affine dynamics
takes the form
\begin{equation}\label{eq:continuity}
    \partial_t \mu_t +
    \nabla \cdot (f(\cdot, u(t,\cdot)) \mu_t) = 0
    \quad \text{in } \mathbb{R}^d \times I,
\end{equation}
where $f(x,u) = f_0(x) + \sum_{i=1}^m u_i f_i(x)$ and
$u: I \times \mathbb{R}^d \rightarrow \mathbb{R}^m$ is a feedback
control law. This equation is understood in the weak sense as
before; for all test functions
$\varphi \in C^\infty_c(\mathbb{R}^d \times I)$,
\begin{equation}\label{eq:weak-continuity}
    \int_I \int_{\mathbb{R}^d} \left[
    \partial_t \varphi(t,x) +
    \langle \nabla_x \varphi(t,x), f(x, u(t,x)) \rangle \right]
    \mu_t(dx) \, dt = 0,
\end{equation}
with appropriate initial condition
$\mu_0 \in \mathcal{P}(\mathbb{R}^d)$ and, when specified,
terminal condition $\mu_T \in \mathcal{P}(\mathbb{R}^d)$.

We say that $(\{\mu_t\}_{t \in I})$ solves the continuity equation
if the above holds and additionally
$u(t,\cdot) \in L^2_{\text{loc}}(\mu_t)$ for almost every $t \in I$,
ensuring the integrability of the vector field along the measure.

\section{Analysis}
It is not immediately obvious that the process $Z_t$ is well defined and its distribution is identical to that of $X_t$. In this section, we establish this correspondence. This result has partially been already been established in \cite{elamvazhuthi2024benamou}, but we develop a more general notion of flow matching by allowing the domain of the interpolating map to be a general latent space $M$. In addition, we also introduce a notion of {\it approximate flow matching.}


\begin{theorem}[Flow Matching with Control System Constraints]
\label{thm:interpogen}
Let $M$ be a separable complete metric space and let $\gamma \in \mathcal{P}(M)$. Suppose $S: M \rightarrow \Omega \subseteq \Gamma \times \mathcal{U}$ is a measurable map defined $\gamma$-almost everywhere on $M$. Define the time-dependent probability measure $\mu_t \in \mathcal{P}(\mathbb{R}^d)$ for Lebesgue-almost every $t \in I$ by 
\[\mu_t := (S^{\omega}_t)_{\#} \gamma.\]

Then $\{\mu_t\}_{t \in I}$ solves the continuity equation 
\begin{equation}
\partial_t \mu_t + \nabla \cdot (f(\cdot, u(t, \cdot))\mu_t) = 0 \quad \text{in } \mathbb{R}^d \times I
\end{equation}
in the weak sense, where the feedback control $u: I \times \mathbb{R}^d \rightarrow \mathbb{R}^m$ is given by
\[u(t,x) = \int_{U} u \, d\eta_{t,x}(u)\]
for Lebesgue-almost every $t \in I$ and $\mu_t$-almost every $x \in \mathbb{R}^d$.

Here, $\eta_{t,x} \in \mathcal{P}(U)$ is the conditional probability measure satisfying the disintegration
\[d(S_t)_{\#} \gamma(x,u) = d\eta_{t,x}(u) \, d\mu_t(x),\]
where $S_t: M \rightarrow \mathbb{R}^d \times U$ is defined by $S_t(m) := (S^{\omega}_t(m), S^u_t(m))$ for $m \in M$.
\end{theorem}
\begin{proof}
Let $g$ denote a generic test function whose domain will be clear from context. Using the measure $\gamma \in \mathcal{P}(M)$ and the map $S: M \rightarrow \Omega$, we define the measure $\eta \in \mathcal{P}(I \times \mathbb{R}^d \times U)$ by
\begin{align}
\int_{I \times \mathbb{R}^d \times U} g(t,x,u) \, d\eta(t,x,u) = \int_{M} \int_{I} g(t,S^\omega_t(m),S^u_t(m)) \, dt \, d\gamma(m)
\end{align}
for all $g \in C_b(I \times \mathbb{R}^d \times U)$. Equivalently, $\eta = F_{\#}(\lambda \otimes \gamma)$ where $F: I \times M \rightarrow I \times \mathbb{R}^d \times U$ is defined by $F(t,m) := (t,S^\omega_t(m),S^u_t(m))$ and $\lambda$ denotes Lebesgue measure on $I$.

We first verify that $\pi^t_{\#}\eta = \lambda$. For any $g \in C_b(I)$,
\begin{align}
\int_{I} g(t) \, d\pi^t_{\#}\eta(t) &= \int_{I \times \mathbb{R}^d \times U} g(\pi^t(t,x,u)) \, d\eta(t,x,u) \\
&= \int_{M} \int_{I} g(t) \, dt \, d\gamma(m) \\
&= \int_{I} g(t) \, dt.
\end{align}

Since $\pi^t_{\#}\eta = \lambda$, there exists a disintegration $\{\eta_t\}_{t \in I}$ with $\eta_t \in \mathcal{P}(\mathbb{R}^d \times U)$ for Lebesgue-almost every $t \in I$ such that
\begin{align}
\int_{I \times \mathbb{R}^d \times U} g(t,x,u) \, d\eta(t,x,u) = \int_{I} \int_{\mathbb{R}^d \times U} g(t,x,u) \, d\eta_t(x,u) \, dt
\end{align}
for all $g \in C_b(I \times \mathbb{R}^d \times U)$.

Since $S^\omega(m)$ solves the control system with control $S^u(m)$, for any $\varphi \in C^1_c(I \times \mathbb{R}^d)$ and $\gamma$-almost every $m \in M$,
\begin{align}
\frac{d}{dt}\varphi(t,S^\omega_t(m)) = \partial_t \varphi(t,S^\omega_t(m)) + \langle \nabla_x \varphi(t,S^\omega_t(m)), f(S^\omega_t(m),S^u_t(m)) \rangle
\end{align}
for almost every $t \in I$. Integrating from $0$ to $T$ and then with respect to $\gamma$,
\begin{align}
&\int_{M} [\varphi(T,S^\omega_T(m)) - \varphi(0,S^\omega_0(m))] \, d\gamma(m) \\
&= \int_{M} \int_{I} [\partial_t \varphi(t,S^\omega_t(m)) + \langle \nabla_x \varphi(t,S^\omega_t(m)), f(S^\omega_t(m),S^u_t(m)) \rangle] \, dt \, d\gamma(m).
\end{align}

Using the definition of $\mu_t = (S^\omega_t)_{\#}\gamma$ and $\eta$, this becomes
\begin{align}
&\int_{\mathbb{R}^d} \varphi(T,x) \, d\mu_T(x) - \int_{\mathbb{R}^d} \varphi(0,x) \, d\mu_0(x) \\
&= \int_{I} \int_{\mathbb{R}^d \times U} [\partial_t \varphi(t,x) + \langle \nabla_x \varphi(t,x), f(x,u) \rangle] \, d\eta_t(x,u) \, dt.
\end{align}

Since $\pi^x_{\#}\eta_t = \mu_t$, there exists a disintegration $\{\eta_{t,x}\}_{x \in \mathbb{R}^d}$ with $\eta_{t,x} \in \mathcal{P}(U)$ for $\mu_t$-almost every $x \in \mathbb{R}^d$ such that
\begin{align}
d\eta_t(x,u) = d\eta_{t,x}(u) \, d\mu_t(x).
\end{align}

Using the affine structure of $f(x,u) = f_0(x) + \sum_{i=1}^m u_i f_i(x)$, we have
\begin{align}
\int_{U} f(x,u) \, d\eta_{t,x}(u) = f\left(x, \int_{U} u \, d\eta_{t,x}(u)\right) = f(x,u(t,x))
\end{align}
where $u(t,x) = \int_{U} u \, d\eta_{t,x}(u)$.

Therefore,
\begin{align}
&\int_{\mathbb{R}^d} \varphi(T,x) \, d\mu_T(x) - \int_{\mathbb{R}^d} \varphi(0,x) \, d\mu_0(x) \\
&= \int_{I} \int_{\mathbb{R}^d} [\partial_t \varphi(t,x) + \langle \nabla_x \varphi(t,x), f(x,u(t,x)) \rangle] \, d\mu_t(x) \, dt.
\end{align}

Finally, we verify that $\mu_t = (S^\omega_t)_{\#}\gamma$ for Lebesgue-almost every $t \in I$. For any $g \in C_b(I \times \mathbb{R}^d)$,
\begin{align}
\int_{I} \int_{\mathbb{R}^d} g(t,x) \, d\mu_t(x) \, dt &= \int_{I \times \mathbb{R}^d \times U} g(t,x) \, d\eta(t,x,u) \\
&= \int_{M} \int_{I} g(t,S^\omega_t(m)) \, dt \, d\gamma(m) \\
&= \int_{I} \int_{\mathbb{R}^d} g(t,x) \, d(S^\omega_t)_{\#}\gamma(x) \, dt.
\end{align}
This completes the proof.
\end{proof}

This immediately gives us the following corollary justifying Algorithm \ref{alg:flowmatching}.

\begin{corollary}[Controlled Flow Matching with Initial and Final Measure Constraints]
\label{cor:flow-matching}
Let $\gamma \in \mathcal{P}_1(\mathbb{R}^d \times \mathbb{R}^d)$ be a transport plan with marginals $\pi^1_{\#}\gamma = \mu_0$ and $\pi^2_{\#}\gamma = \mu_T$. Suppose $S: \mathbb{R}^d \times \mathbb{R}^d \rightarrow \Omega$ is a measurable map defined $\gamma$-almost everywhere such that 
\[S^{\omega}_0(x_0,x_T) = x_0, \quad S^{\omega}_T(x_0,x_T) = x_T\]
for $\gamma$-almost every $(x_0,x_T) \in \mathbb{R}^d \times \mathbb{R}^d$. 

Define the time-dependent probability measure $\mu_t \in \mathcal{P}(\mathbb{R}^d)$ for Lebesgue-almost every $t \in I$ by 
\[\mu_t := (S^{\omega}_t)_{\#}\gamma.\]

Then $\{\mu_t\}_{t \in I}$ solves the continuity equation 
\begin{equation}
\partial_t \mu_t + \nabla \cdot (f(\cdot, u(t,\cdot))\mu_t) = 0 \quad \text{in } \mathbb{R}^d \times I
\end{equation}
in the weak sense, with initial condition $\mu_0 = \pi^1_{\#}\gamma$ and terminal condition $\mu_T = \pi^2_{\#}\gamma$. The feedback control $u: I \times \mathbb{R}^d \rightarrow \mathbb{R}^m$ is given by
\[u(t,x) = \int_{U} u \, d\eta_{t,x}(u)\]
for Lebesgue-almost every $t \in I$ and $\mu_t$-almost every $x \in \mathbb{R}^d$, where $\eta_{t,x} \in \mathcal{P}(U)$ is the conditional probability measure satisfying
\[d(S_t)_{\#}\gamma(x,u) = d\eta_{t,x}(u) \, d\mu_t(x),\]
with $S_t: \mathbb{R}^d \times \mathbb{R}^d \rightarrow \mathbb{R}^d \times U$ defined by $S_t(x_0,x_T) := (S^{\omega}_t(x_0,x_T), S^u_t(x_0,x_T))$.
\end{corollary}



The above result indicates that we can potentially realize a process $Z_t$ for which the control $u(t,x)$ transports the system from $\mu_0$ to $\mu_T$. However, it is not immediate that such a process exists since the control $u(t,x)$ might not be Lipschitz. In existing literature in machine learning \cite{albergo2023stochastic,liu2022flow,lipmanflow}, extra assumptions are made in order to construct such a process. In the following result, the superposition principle from optimal transport theory \cite[Theorem 3.4]{ambrosio2014continuity} enables the construction of such a measure without Lipschitz property. On the other hand, the vector field could be extremely irregular, potentially causing non-unique solutions.

\begin{theorem}[Superposition Principle {\cite{ambrosio2014continuity}}]
\label{thm:superposition}
Let $v: I \times \mathbb{R}^d \rightarrow \mathbb{R}^d$ be a Borel measurable vector field. Suppose $\{\mu_t\}_{t \in I}$ with $\mu_t \in \mathcal{P}(\mathbb{R}^d)$ is a weak solution to the continuity equation
\begin{equation}
\partial_t \mu_t + \nabla \cdot (v(t,x) \mu_t) = 0 \quad \text{in } \mathbb{R}^d \times I
\end{equation}
satisfying the integrability condition
\[\int_I \int_{\mathbb{R}^d} \frac{|v(t,x)|}{1+|x|} \, d\mu_t(x) \, dt < \infty.\]

Then there exists a probability measure $\mathbb{P} \in \mathcal{P}(\Gamma)$ concentrated on solutions of the ODE
\begin{equation}
\dot{\omega}(t) = v(t,\omega(t)), \quad \omega(0) = \omega_0
\end{equation}
for $\mathbb{P}$-almost every $\omega \in \Gamma$, such that $(e_t)_{\#}\mathbb{P} = \mu_t$ for all $t \in I$.
\end{theorem}

Given this theorem, we can state the following.

\begin{theorem}[Realization of Process]
\label{thm:realization}
Let $\gamma \in \mathcal{P}(M)$ be a probability measure on a separable complete metric space $M$. Suppose $S: M \rightarrow \Omega$ is a measurable map defined $\gamma$-almost everywhere on $M$. Additionally, assume that 
\begin{equation}
\int_{M} \int_I |S^u_t(m)| \, dt \, d\gamma(m) < \infty.
\end{equation}

Define $\mu_t := (S^{\omega}_t)_{\#}\gamma$ for $t \in I$. Then $\{\mu_t\}_{t \in I}$ solves the continuity equation  
\begin{equation}
\partial_t \mu_t + \nabla \cdot (f(\cdot, u(t,\cdot))\mu_t) = 0 \quad \text{in } \mathbb{R}^d \times I
\end{equation}
in the weak sense, with initial condition $\mu_0 = \pi^1_{\#}\gamma$ and terminal condition $\mu_T = \pi^2_{\#}\gamma$. with feedback control
\[u(t,x) = \int_{U} u \, d\eta_{t,x}(u)\]
where $\eta_{t,x} \in \mathcal{P}(U)$ is the conditional distribution of $U_t = S^u_t(M)$ given $X_t = x$.

Then there exists a probability measure $\mathbb{P} \in \mathcal{P}(\mathbb{R}^d \times \Gamma)$ such that 
\begin{equation}
\dot{\omega}(t) = f(\omega(t), u(t,\omega(t))), \quad \omega(0) = y
\end{equation}
for $\mathbb{P}$-almost every $(y,\omega) \in \mathbb{R}^d \times \Gamma$, and $(e_t)_{\#}\mathbb{P} = \mu_t$ for all $t \in I$.
\end{theorem}

\begin{proof}
As in the proof of Theorem \ref{thm:interpogen}, let $\eta := F_{\#}(\lambda \otimes \gamma)$ be the pushforward of $\lambda \otimes \gamma$ under the map $F: I \times M \rightarrow I \times \mathbb{R}^d \times U$ defined by $F(t,m) := (t,S^\omega_t(m),S^u_t(m))$, where $\lambda$ denotes Lebesgue measure on $I$. We compute
\begin{align*}
\int_M \int_I |S^u_t(m)| \, dt \, d\gamma(m) 
&= \int_{I \times \mathbb{R}^d \times U} |u| \, d(F)_{\#}(\lambda \otimes \gamma)(t,x,u) \\
&= \int_{I \times \mathbb{R}^d \times U} |u| \, d\eta(t,x,u)\\
&= \int_I \int_{\mathbb{R}^d} \int_{U} |u| \, d\eta_{t,x}(u) \, d\mu_t(x) \, dt \\
&= \int_I \int_{\mathbb{R}^d} \left(\int_{U} |u| \, d\eta_{t,x}(u)\right) d\mu_t(x) \, dt \\
&= \int_I \int_{\mathbb{R}^d} |u(t,x)| \, d\mu_t(x) \, dt 
\end{align*}

This implies
\[\int_I \int_{\mathbb{R}^d} |u(t,x)| \, d\mu_t(x) \, dt < \infty.\]

Therefore, by Assumption \ref{asmp:sublin},
\begin{align*}
\int_I \int_{\mathbb{R}^d} \frac{|f(x,u(t,x))|}{1+|x|} \, d\mu_t(x) \, dt 
&\leq \int_I \int_{\mathbb{R}^d} \frac{M(1 + |x|)}{1+|x|} \, d\mu_t(x) \, dt \\
&\quad + M\sum_{i=1}^m \int_I \int_{\mathbb{R}^d} \frac{(1 + |x|)|u_i(t,x)|}{1+|x|} \, d\mu_t(x) \, dt \\
&= MT + M\sum_{i=1}^m \int_I \int_{\mathbb{R}^d} |u_i(t,x)| \, d\mu_t(x) \, dt < \infty.
\end{align*}

The result now follows from the superposition principle (Theorem \ref{thm:superposition}).
\end{proof}

From the above result, there exists a process $Z_t$ with law $\mathbb{P}$ such that $\mathbb{P}(Z_t \in A) = \mu_t(A)$ for all measurable sets $A \subseteq \mathbb{R}^d$.

A useful property of the constructed control is that if $S^u_t(m) \in U$ for all $m \in M$ and $t \in I$, where $U$ is a convex compact set, then $u(t,x) \in U$ for all $(t,x)$. This allows one to ensure control constraints are satisfied.

Many times one might not have an exact map $S(x_0,x_T)$ that steers $x_0$ to $x_T$ exactly, but only approximately. Towards this end, we establish a result on constructing approximate transport plans using approximate point-to-point steering controls.

\begin{theorem}[\textbf{Approximate Flow Matching}]
\label{thm:approximate-flow}
Let $\mu_0, \mu_T \in \mathcal{P}_1(\mathbb{R}^d)$ be compactly supported probability measures, and let $\gamma \in \Pi(\mu_0, \mu_T)$ be a transport plan. Suppose $S: \mathbb{R}^d \times \mathbb{R}^d \rightarrow \Omega$ is a measurable map defined $\gamma$-almost everywhere such that 
\[S^{\omega}_0(x_0,x_T) = x_0, \quad |S^{\omega}_T(x_0,x_T) - x_T| < \varepsilon\]
for $\gamma$-almost every $(x_0,x_T) \in \mathbb{R}^d \times \mathbb{R}^d$. Additionally, assume that 
\begin{equation}
\int_{\mathbb{R}^d \times \mathbb{R}^d} \int_{I} |S^u_t(x_0,x_T)|^2 \, dt \, d\gamma(x_0,x_T) < \infty.
\end{equation}

Then there exists a pair $(\{\mu_t\}_{t \in I}, u)$ that solves the continuity equation \eqref{eq:continuity} in the weak sense, with $\mu_0 = \pi^1_{\#}\gamma$ and $u(t,\cdot) \in L^2(\mu_t)$ for almost every $t \in I$, where
\[u(t,x) = \int_{U} u \, d\eta_{t,x}(u).\]

Moreover, defining $\tilde{\mu}_T := (S^{\omega}_T)_{\#}\gamma$, we have
\[W_2(\tilde{\mu}_T, \mu_T) \leq \varepsilon.\]
\end{theorem}

\begin{proof}
The first part of the result follows as in the case of exact flow matching. We establish the bound on the Wasserstein metric. Note that $\tilde{\mu}_T = (S^\omega_T)_{\#}\gamma$. Since $\gamma \in \Pi(\mu_0, \mu_T)$, we have $\pi^2_{\#}\gamma = \mu_T$. Hence,
\begin{align}
W^2_2(\tilde{\mu}_T, \mu_T) &= W^2_2((S^\omega_T)_{\#}\gamma, \pi^2_{\#}\gamma) \\
&\leq \int_{\mathbb{R}^d \times \mathbb{R}^d} |S^\omega_T(x_0,x_T) - x_T|^2 \, d\gamma(x_0,x_T) \\
&< \varepsilon^2.
\end{align}
\end{proof}

There are situations when one wants to steer probability measures not on the entire state space, but only on a subset of the configurations through an output map. Toward this end, we develop a notion of output flow matching.

\begin{theorem}[\textbf{Output Flow Matching}]
\label{thm:output-flow}
Let $h: \mathbb{R}^d \rightarrow \mathbb{R}^o$ be a measurable output map. Let $\mu_0 \in \mathcal{P}(\mathbb{R}^d)$ and $\nu_T \in \mathcal{P}_1(\mathbb{R}^o)$ be compactly supported measures. Let $\gamma \in \Pi(\mu_0, \nu_T)$ be a transport plan.

Suppose $S: \mathbb{R}^d \times \mathbb{R}^o \rightarrow \Omega$ is a measurable map defined $\gamma$-almost everywhere such that 
\[S^{\omega}_0(x, y) = x, \quad h(S^{\omega}_T(x, y)) = y\]
for $\gamma$-almost every $(x, y) \in \mathbb{R}^d \times \mathbb{R}^o$. Additionally, assume that 
\begin{equation}
\int_{\mathbb{R}^d \times \mathbb{R}^o} \int_{I} |S^u_t(x, y)|^2 \, dt \, d\gamma(x, y) < \infty.
\end{equation}

Then there exists a pair $(\{\mu_t\}_{t \in I}, u)$ that solves the continuity equation \eqref{eq:continuity} in the weak sense, with $\mu_0 = \pi^1_{\#}\gamma$ and $u(t,\cdot) \in L^2(\mu_t)$ for almost every $t \in I$, where
\[u(t,x) = \int_{U} u \, d\eta_{t,x}(u).\]

Moreover, $h_{\#}\mu_T = \nu_T = \pi^2_{\#}\gamma$.
\end{theorem}

\begin{proof}
Since $\gamma \in \Pi(\mu_0, \nu_T)$ is already a transport plan between the state space and output space, we can use it directly without lifting.

Define $\mu_t := (S^\omega_t)_{\#}\gamma$ for $t \in I$. By assumption, $\mu_0 = (S^\omega_0)_{\#}\gamma = \pi^1_{\#}\gamma$ since $S^\omega_0(x,y) = x$.

For the terminal time, we have:
\begin{align}
h_{\#}\mu_T &= h_{\#}(S^\omega_T)_{\#}\gamma \\
&= (h \circ S^\omega_T)_{\#}\gamma \\
&= \pi^2_{\#}\gamma \\
&= \nu_T,
\end{align}
where the second-to-last equality uses that $h(S^\omega_T(x,y)) = y$ for $\gamma$-almost every $(x,y)$.

The existence of the control $u$ satisfying the continuity equation follows from Theorem \ref{thm:interpogen} applied to the measure $\gamma$ and map $S$.
\end{proof}

In other words, even if exact matching is impossible, we can construct approximate controls that achieve $\varepsilon$-closeness in Wasserstein distance.

\section{Stabilization to sets using Flow matching and Time-reversals}

So far we have developed flow-matching constructions for transporting probability measures under control-affine dynamics, with exact, approximate, and output-based variants. These tools provide scalable ways to steer distributions of states between prescribed initial and terminal measures.  

We now turn to a complementary but more fundamental problem in control. That is \emph{steering to points or target sets}. Rather than transporting one distribution to another, the objective here is to design controllers that drive trajectories toward a desired equilibrium or invariant set. Our approach builds on the same flow-matching philosophy, but incorporates a new element in terms of the time-reversal viewpoint inspired by denoising diffusion models \cite{song2020score}.  

In this section we reinterpret stabilization as a denoising task: suitable excitations are introduced to ``noise'' the system, and time reversal of this process yields feedback laws that achieve steering to the target set. We present two specific constructions of such noising processes. One based on randomized controls and one based on PMP extremals, and analyze their properties. These methods extend the flow-matching framework to stabilization, bridging measure-transport and classical control.

\subsection{Flow Matching using Randomized Controls and Time-reversal}

In this section, we introduce the randomized controls based approach to noise the system taking inspiration from Denoising diffusion based approach to steering a system to a target configuration \cite{elamvazhuthi2025score,mei2025flow}. However, rather than injecting the system with white noise, we use probability measures on the set of controls, to noise the system. This provides a regular alternative to noise and denoise the system. See Figure \ref{fig:flow_matching_schematic}
 for a visualization of the idea and Algorithm \ref{alg4} for the corresponding implementation.

\begin{figure}[t]
    \centering
    
    \includegraphics[width=0.9\linewidth]{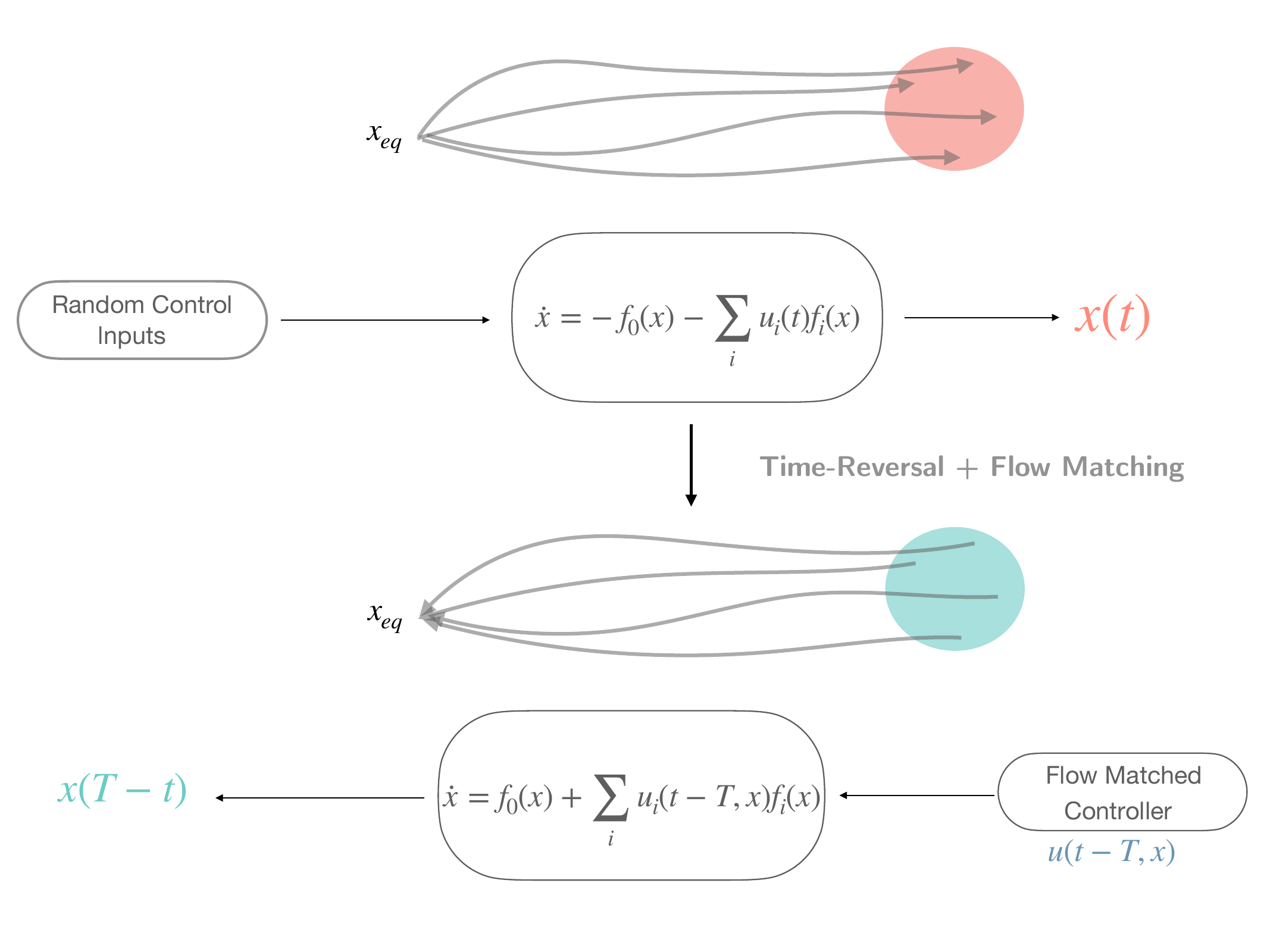}
    
    \caption{
        \textbf{Randomized Control Sampling Flow matching.}
        Random control inputs applied to the time-reversed system 
        produce noisy trajectories (\textit{noising process}, top). 
        The learned reversed controller denoises the system, 
        guiding trajectories back to desired states 
        (\textit{denoising process}, bottom). 
        The red ball illustrates the evolution of a single system state $x(t)$ 
        under both processes.
    }
    \label{fig:flow_matching_schematic}
\end{figure}

\begin{algorithm}
\caption{Flow Matching with Noisy Control}
\begin{algorithmic}[1]
\State  Fix compact target set $\Omega \subseteq \Rd$. Sample initial condition $x  \sim \mu_0$, where $\mu_0$ is supported on $\Omega$. 

\State Independently, sample $\Xi \sim \nu_0$ and a standard Brownian motion $(B_t)_{t\in[0,T]}$.

\State Define the interpolation path $X_t$ and the corresponding control
\[
U_t=\Xi+B_t
\]
using the map $S_{X_0}(U)$ defined according to \eqref{eq:ctrsys}.
\State Learn control law $u: I \times \mathbb{R}^d \rightarrow \mathbb{R}^m$ by solving the regression problem:
\[\min_{u} \int_I \mathbb{E}\left[\|U_t - u(t, X_t)\|^2\right] dt\]

\State Sample new data via learned flow.
  Sample from $Z_T$ according to some probability distribution and solve the ODE in reverse according to:
    \[
    \dot{Z}_t = f_0(Z) + \sum_{i=1}^m u_i(t, Z_t) f_i(Z_t)
    \]
    or equivalently  sample from $\tilde{Z}_0$ according to some probability distribution and solve the ODE:
    \[
    \dot{\tilde{Z}}_t = -f_0(\tilde{Z}) - \sum_{i=1}^m u_i(T-t, \tilde{Z}_t) f_i(Z_t)
    \]
    \State Output $Z_0$.
\end{algorithmic}
\label{alg4}
\end{algorithm}

\begin{theorem}
   \label{thm:randcongr}
    \textbf{(Controlled Flow Matching with Random Control)}
Consider the probability measure $\gamma \in \mathcal{P}(\mathcal{U})$. Suppose $S: \mathcal{U} \rightarrow \Gamma \times \mathcal{U}$ is a measurable map defined $\gamma$ almost everywhere on $M$, such that $S^u_t(\alpha) = \alpha$ and $S^{\omega}_0(\alpha) = x$ for $\gamma$ almost every $u \in \mathcal{U}$.
Additionally, assume that 
\begin{equation}
\label{eq:finengy}
\int_{\mathcal{U}}\int_{0}^T |S^u_t(m)|^2dt d\gamma(m) <\infty
\end{equation}
Let $\mu_t$ be the distribution of $S^{\omega}_t(Y)$.
That is, $\mu_t : = (S^{\omega}_t)_{\#} \gamma$. Then $\mu_t$ solves the continuity equation \eqref{eq:ctnteq}.  for almost every $t \in I$, given by
\[u(t,x) = \int_{U} u d\eta_{t,x}(u)\]
where $t \mapsto \eta_{t,x}(u)$ is the conditional distribution  $U_t = S^u_t(Y)$ given $X_t = x$.
\end{theorem}

An interesting question in this case is what are good choices of probability measures $\mathcal{P}(\mathcal{U})$, on the set of controls . In denoising diffusion \cite{song2020score}, one can loosely say that the choice is white noise. However, as demonstrated in this section, one can take more regular controls. For instance, if one takes $B(t)$ to be the Brownian motion on $\Rd$. Then $B$ defines a measure on $C([0,T];\Rd)$ known as the {\it Wiener measure} \cite{bogachev2010differentiable}.  Then one can prove a regularity result for the flow $\mu_t$. Toward this end, we will need the following definition.

\begin{definition}[End Point Map of time-reversed system]
\label{def:endpoint}
We define the end point map $E: \mathcal{U} \rightarrow \mathbb{R}^d$ by $E(u) = \omega(T)$, where $\omega \in \Gamma$ solves the control system with control $u \in \mathcal{U}$, i.e., $\dot{\omega}(t) = -f(\omega(t), u(t))$ with $\omega(0) = x$.
\end{definition}

The surjectivity of the differential of this map ensures that small variations in the control produce variations that span the entire state space, preventing the formation of singular curves. A common assumption that is required in optimal control theory is this surjectivity property.

\begin{assumption}[Non-existence of Singular Curves]
\label{asmp:nonsingular}
For every $u \in \mathcal{U}$, the differential $dE(u): L^2(I; \mathbb{R}^m) \rightarrow \mathbb{R}^d$ is surjective.
\end{assumption}

Additionally, to characterize the support of the measures we will need the following natural controllability assumption on the range of the end point map.

\begin{definition}[Global Controllability of Time-Reversed System]
\label{def:controllable}
Given $x \in \mathbb{R}^d$, the system is globally controllable from $x$ at time $T$ if for every $\tilde{x} \in \mathbb{R}^d$, there exists a control $u \in L^2(I; \mathbb{R}^m)$ such that the solution satisfies $\omega(0) = x$ and $\omega(T) = \tilde{x}$. 
\end{definition}

Given these definitions and assumptions we have the following result on the support and regularity properties of $\mu_t$.

\begin{theorem}
Given Assumption \ref{asmp:sublin}. Let $\nu_0 \in \mathcal P_2(\mathbb R^m)$, and let
$U_t = \Xi + B_t, \quad t\in[0,T]$,
where $B$ is a standard Brownian motion in $\mathbb R^m$ and $\Xi$ is an
$\mathbb R^m$-valued random variable, independent of $B$, with law $\nu_0 \in \mathcal{P}_(\mathbb{R}^m)$.
Let $\gamma_{\nu_0} \in \mathcal P(C([0,T];\mathbb R^m))$ be the law of $U$.

Then $\gamma_{\nu_0}$ satisfies the hypotheses of Theorem \ref{thm:randcongr}.
If $\mu_t$ denotes the corresponding state law from Theorem \ref{thm:randcongr},
then for every $t\in(0,T]$,
\[
\operatorname{supp}\mu_t = \overline{R^{\nu_0}_{t,x}},
\]
where
\[
R^{\nu_0}_{t,x}
:=
\Bigl\{
\omega(t) \;:\;
(\omega,u)\in \Omega_t,\;
u\in C([0,t];\mathbb R^m),\;
u(0)\in \operatorname{supp}\nu_0
\Bigr\},
\]
and $\Omega_t$ denotes the set of admissible state-control pairs on $[0,t]$.

If the time-reversed system is globally controllable from $x$ at time $t$,
then $
\operatorname{supp}\mu_t = \mathbb R^d.$

Moreover, if Assumption \ref{asmp:nonsingular} holds, then $
\mu_t$ is absolutely continuous with respect to the Lebesgue measure for all $ t\in(0,T]$
\end{theorem}
\begin{proof}
First, since $\nu_0 \in \mathcal{P}(\R^m)$ and $B$ has finite second moments on $[0,T]$,
\[
\int_{C([0,T];\mathbb R^m)} \int_0^T |u(s)|^2\,ds \, d\gamma_{\nu_0}(u)
=
\mathbb E\!\left[\int_0^T |U_s|^2\,ds\right]
< \infty.
\]
Hence $\gamma_{\nu_0}$ satisfies the integrability condition in
Theorem \ref{thm:randcongr}.

For each $a\in\mathbb R^m$, let $\gamma_a$ denote the law of the process
$a+B$. Then
$
\gamma_{\nu_0} = \int_{\mathbb R^m} \gamma_a \, d\nu_0(a).$
Moreover, \[
\operatorname{supp}\gamma_a
=
\{u\in C([0,T];\mathbb R^m): u(0)=a\},\]
so
\[
\operatorname{supp}\gamma_{\nu_0}
=
\{u\in C([0,T];\mathbb R^m): u(0)\in \operatorname{supp}\nu_0\}.
\]

Let $E_t:C([0,t];\mathbb R^m)\to\mathbb R^d$ be the end-point map defined by
$E_t(u)=\omega(t)$, where $\omega$ solves
\[
\dot\omega(s)=-f(\omega(s),u(s)),\qquad \omega(0)=x.
\]
By continuous dependence of solutions on controls, $E_t$ is continuous.
Therefore,
\[
\operatorname{supp}\mu_t
=
\operatorname{supp}(E_t)_{\#}\gamma_{\nu_0}
=
\overline{E_t(\operatorname{supp}\gamma_{\nu_0})}
=
\overline{R^{\nu_0}_{t,x}}.
\]

If the time-reversed system is globally controllable from $x$ at time $t$,
then the reachable set under square integrable controls, $L^2([0,t];\mathbb R^m)$, is all of $\mathbb R^d$.
Fix any $a\in\operatorname{supp}\nu_0$. Continuous controls satisfying
$u(0)=a$ are dense in $L^2([0,t];\mathbb R^m)$, since one may approximate
any square integrable control by continuous controls and then modify it on a small
initial interval to enforce the value $u(0)=a$ without changing the
$L^2$ norm significantly. Hence, by continuity of the end-point map,
$\overline{R^{\nu_0}_{t,x}}=\mathbb R^d$, proving that
$\operatorname{supp}\mu_t=\mathbb R^d$. Here, we have used the fact that sample paths of Brownian motion are dense in the space of continuous functions with initial value equal to $0$ \citep[Page 30]{bogachev2010differentiable}.

The absolute continuity of $\mu_t$ follows from Assumption \eqref{asmp:nonsingular} and \citep[9.2.5. Corollary]{bogachev2010differentiable}. Note that the condition required in the cited corollary that $\gamma$ be ``continuous along vectors from a dense set" follows from \citep[3.1.9. Theorem]{bogachev2010differentiable}.  For each $a \in \mathbb{R}^d$, $
(E_t)_{\#}\gamma_a $ is absolutely continuous with respect to the Lebesgue measure. 
Since
\[
\mu_t(A)
=
(E_t)_{\#}\gamma_{\nu_0}(A)
=
\int_{\mathbb R^m} (E_t)_{\#}\gamma_a(A) \, d\nu_0(a),
\]
for all Borel measurable sets $A \subseteq C([0,T];\mathbb{R}^d)$, the absolute continuity of $\mu_t$  follows.
\end{proof}

An interesting straightforward extension of the randomized-control construction is to impose \textbf{hard control constraints} directly at the level of the sampling law. Let $K \subset \mathbb R^m$ be a closed convex set, and let $\Pi_K : \mathbb R^m \to K$ denote the metric projection. Instead of sampling the control from unconstrained Brownian motion, one may take
\[
U_t := \Pi_K(B_t),
\]
where $B$ is a Brownian motion in $\mathbb R^m$ with initial condition chosen in $K$. Since $\Pi_K$ is continuous and satisfies $\Pi_K(v)=v$ for every $v\in K$, the law of the process $U$ is supported on the constrained path space
\[
\mathcal C_K := \{u \in C([0,T];\mathbb R^m) : u(t)\in K \text{ for all } t\in[0,T]\}.
\]
In fact, this support is exactly $\mathcal C_K$ and every constrained continuous path is fixed by the projection map, while the support of Brownian motion is the full space of continuous paths with the prescribed initial value. Thus, projected Brownian controls provide a natural way to randomize only over admissible controls.

This fits naturally into the flow-matching framework of Theorem \ref{thm:randcongr} since if all sampled controls take values in a convex compact set, then the averaged feedback law $u(t,x) : =\int ud\eta_x (u)$ obtained by conditional expectation also takes values in that same set. Consequently, the learned controller remains admissible at every time.

From the point of view of the induced state measures, the support of $\mu_t$ is then the closure of the reachable set generated by constrained continuous controls, by continuity of the end-point map together with the support characterization of pushforwards. In other words, replacing Brownian motion by its projection onto $K$ restricts the noising process to trajectories generated by controls in $K$, and therefore restricts the marginals $\mu_t$ to the corresponding constrained reachable region.

If, in addition, $K$ is compact, then the controls are uniformly bounded. Under Assumption \ref{asmp:sublin}, this implies that every corresponding state trajectory remains in a bounded set depending only on $x_0$, $T$, and $K$. Hence the support of $\mu_t$ is contained in a compact neighborhood of $x_0$. This gives a simple heuristic for a stabilizing effect of the time-reversed system. That is, if the forward noising dynamics remain confined to a controlled neighborhood of $x_0$, the reversed dynamics are expected to steer trajectories back toward that same region, and more generally toward the target set around which the noising process was constructed. 

This viewpoint provides a natural bridge from the support and localization properties of the forward noising process to a practical reverse-time stabilization scheme. In particular, once the forward randomized dynamics have been constructed and the corresponding feedback law has been learned, the reverse-time system may be viewed as transporting mass back toward the target region encoded by $\mu_0$. This leads to the following procedure.
\begin{enumerate}
\item Pick $\mu_0$ that is supported on the target set that is required to be stabilized.

\item Choose a randomized control law $\gamma$ on the control space (for instance Wiener measure, or a constrained variant), generate the noising process $(X_t,U_t)$, and learn the averaged feedback law $u(t,x)$ from the regression problem.

\item If $\mu_T$ is difficult to sample from exactly, replace it by any convenient alternative measure $\tilde{\mu}_T$ such that $\tilde{\mu}_T$  is absolutely continuous with respect to $\mu_T$.

\item For the time-reversed flow, by Proposition \ref{prop:change-initial}, one may replace $\mu_T$ by $\tilde{\mu}_T$ without changing the reverse dynamics. Hence the terminal support of the reversed process is contained in $\operatorname{supp}\mu_0$. In particular, if $\mu_0 = \delta_{x_{\mathrm{eq}}}$, then the time-reversed system is transferred back to the equilibrium point $x_{\mathrm{eq}}$.
\end{enumerate}

\textbf{Model-free implementation:} We note that the randomized-control method developed here admits a model-free implementation, provided one has access to a simulator capable of generating trajectories of the time-reversed system under sampled control inputs. In particular, the training data required for the regression step can be obtained directly from simulated trajectory--control pairs, without requiring an explicit analytical description of the dynamics.

\textbf{General Noising:} Injecting Brownian motion through the control channels is only one way to noise the system. In general, one can take any sampling algorithm that explores the state-space and use the noised trajectories to learn a feedback control to denoise the system using flow matching. One possible approach to noise a fully actuated system, for instance, is to use path planners from the robotics literature. In the numerical simulation section, we use the RRT* algorithm \cite{karaman2011sampling} provide training trajectories to learn the feedback control that navigates around obstacles. See section \ref{sec:rrtstar} for a demonstration of this idea.

\subsection{Flow Matching on PMP extremals}

This subsection develops a PMP-based approach to flow matching. 
Instead of relying on generic interpolation curves, we construct trajectories directly from solutions of an optimal control problem. 
Pontryagin’s Maximum Principle (PMP) provides first-order necessary conditions for optimality by coupling the state dynamics with an adjoint (costate) system, yielding extremals that encode the structure of optimal trajectories for a given cost functional.

By sampling over different initial costates, one obtains a family of PMP extremals that can serve as interpolation curves for flow matching. This connects the regression-based construction of vector fields with variational principles from optimal control, ensuring that the induced flows inherit optimality properties under suitable assumptions.  

This viewpoint is particularly useful for stabilization problems. That is, through time-reversal, PMP extremals can be repurposed to generate trajectories that optimally steer initial states toward a target set. In what follows, we formalize this construction by introducing flow matching based on PMP extremals. We consider two different settings. One using PMP systems for fixed end-point problems, and another for variable end-point problems with a terminal cost.

\begin{figure}[t]
    \centering
    
    \includegraphics[width=0.9\linewidth]{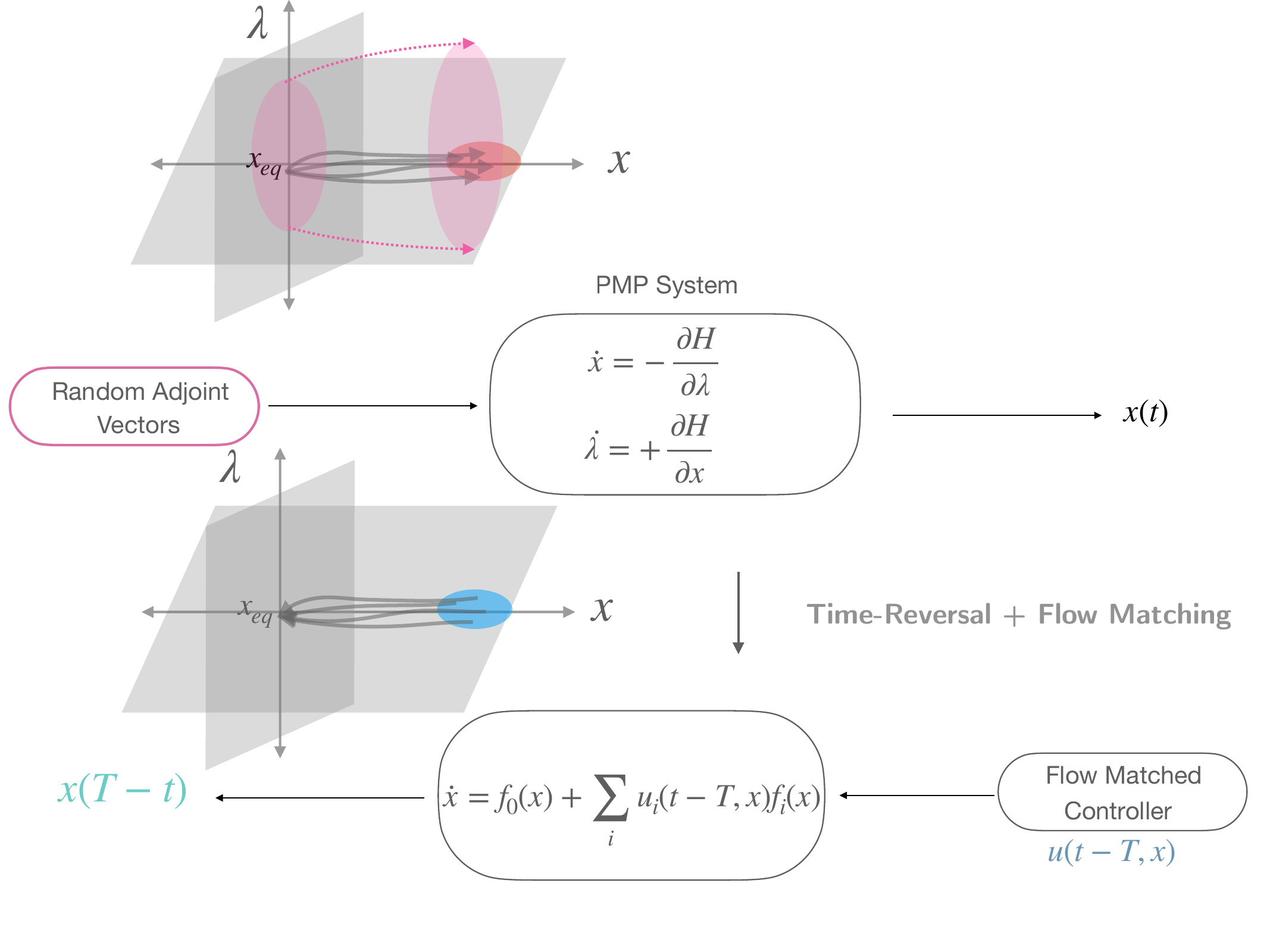}
    
    \caption{
        \textbf{PMP based Flow Matching.}
        In this formulation, random adjoint vectors $p$ are initialized for the time-reversed Pontryagin system to generate noisy dynamics.  The red phase depicts stochastic adjoint perturbations (noising), 
        while the blue phase represents controlled stabilization (denoising)
        toward the equilibrium $x_{\mathrm{eq}.}$
    }
    \label{fig:flow_matching_schematic2}
\end{figure}

\subsection{Fixed End-point Problems}

Let $L: \mathbb{R}^d \times \mathbb{R}^m \rightarrow \mathbb{R}$ be a smooth cost function, and consider the Pontryagin system for the time-reversed control dynamics:

\begin{align}
\label{eq:pmp}
\dot{\omega}(t) &= -f(\omega(t), \alpha(\omega(t), p(t))) \\
\dot{p}(t) &= \langle \nabla_x f(\omega(t), \alpha(\omega(t), p(t))), p(t) \rangle + \nabla_x L(\omega(t), \alpha(\omega(t), p(t))) \nonumber \\
\omega(0) &= x_0 \nonumber \\
p(0) &= p_0 \nonumber \\
\alpha(\omega, p) &= \arg\min_{\alpha \in \mathbb{R}^m} \left[-\langle p, f(\omega, \alpha) \rangle + L(\omega, \alpha)\right] \nonumber
\end{align}

It is well known that this system is the time-reversed Hamiltonian system \cite{cannarsa2008semiconcavity}, that provides necessary condition of optimality for the optimal control problem.

\begin{align}
\label{eq:nonlinear-opt}
\inf_{x,u} &\int_I L(x(t),u(t)) \, dt \\
\text{subject to } &\dot{x}(t) = f(x,u) \nonumber \\
&x(0) = \omega(T) , \quad x(T) = x_0 \nonumber
\end{align}

Our goal is to use the PMP system to noise the control system onto the reachable space. By varying the initial costate, we generate a family of PMP extremals that serve as interpolation curves for flow matching. 
In this way, regression-based learning of vector fields is linked to variational principles of optimal control, and the resulting flows can inherit optimality properties under suitable assumptions. 
The practical implementation of this procedure is summarized in Algorithm~\ref{alg:pmp-flow} and a conceptual visualization of the idea is presented in Figure \ref{fig:flow_matching_schematic2}.

\begin{algorithm}
\caption{Flow Matching Along Pontryagin System through Time Reversal (Fixed-end Point Problems)}
\label{alg:pmp-flow}
\begin{algorithmic}[1]
\State Fix compact target set $\Omega \subseteq \Rd$. Sample initial condition $x  \sim \mu_0$, where $\mu_0$ is supported on $\Omega$. 
\State Sample adjoint vector $p_0 \sim \mu_p$ where $\mu_p \in \mathcal{P}(\mathbb{R}^d)$.
\State Define the interpolation path $X_t$ and corresponding control $U_t$ using the map $S: \mathbb{R}^d \times \Rd \rightarrow \Omega$ defined by solving \eqref{eq:pmp}.
\State Learn control law $u: I \times \mathbb{R}^d \rightarrow \mathbb{R}^m$ by solving the regression problem:
\[\min_{u} \int_I \mathbb{E}\left[\|U_t - u(t, X_t)\|^2\right] dt\]
\State Sample new data via learned flow. Sample $Z_T \sim \mu_T$ for some $\mu_T \in \mathcal{P}(\mathbb{R}^d)$ and solve the ODE in reverse:
\[\dot{Z}_t = f(Z_t, u(t, Z_t)) = f_0(Z_t) + \sum_{i=1}^m u_i(t, Z_t) f_i(Z_t)\]
or equivalently, sample $\tilde{Z}_0 \sim \mu_0$ and solve forward:
\[\dot{\tilde{Z}}_t = -f(\tilde{Z}_t, u(T-t, \tilde{Z}_t)) = -f_0(\tilde{Z}_t) - \sum_{i=1}^m u_i(T-t, \tilde{Z}_t) f_i(\tilde{Z}_t)\]
\State \textbf{Output:} $Z_0$ (from reverse integration) or $\tilde{Z}_T$ (from forward integration)
\end{algorithmic}
\end{algorithm}

To elaborate on the algorithm, we define a family of interpolants as follows. Define $S: \Rd \times \mathbb{R}^d \rightarrow \Omega$ by $S(x_0,p_0) = (\omega, u)$ where $(\omega, p)$ solves the Pontryagin system \eqref{eq:pmp} with initial conditions $\omega(0) = x_0$ and $p(0) = p_0$.

Using these equations we define the forward process in the following way. Let $X_0 \sim \mu_0$ and $P_0 \sim \mu_p$ where $\mu_0$ is the distribution of initial conditions and $\mu_p \in \mathcal{P}(\mathbb{R}^d)$ is a distribution over initial adjoint vectors. Define the stochastic processes $(X_t, U_t) := (S^\omega_t(X_0,P_0), S^u_t(X_0,P_0))$ for $t \in I$. For this process we want to construct the averaged process. The idea is that though we don't know $\mu_t$ exactly for any $t \in (0,T)$, $\mu_t$ is supported on the reachable set. Later we will show how this can be combined with time-reversals to stabilize a given target set $\Omega$.

In order to be able to construct the averaged forward process corresponding to this PMP based interpolation we need a number of assumptions. We start with the following regularity assumption.

\begin{assumption}
\label{asmp:pmp-control}
The map $\mathbb{R}^d \times \mathbb{R}^d \ni (\omega, p) \mapsto \alpha(\omega,p) \in \mathbb{R}^m$ defined by
\[\alpha(\omega,p) := \arg\min_{\alpha \in \mathbb{R}^m} \left[\langle p, f(\omega,\alpha) \rangle + L(\omega,\alpha)\right]\]
is uniquely defined and locally Lipschitz continuous. Moreover, there exists a constant $M > 0$ such that 
\[|\alpha(\omega,p)| \leq M(1 + |\omega| + |p|)\]
for all $(\omega, p) \in \mathbb{R}^d \times \mathbb{R}^d$.
\end{assumption}

 Proposition \ref{prop:relaass} shows this Assumption to be true under the additional assumption that $L$ is strongly convex in 
$u$, uniformly in $x$, the control vector fields $f_i(x)$ are globally bounded, and the following assumptions on the cost function.

\begin{assumption}
\label{asmp:cost}
The cost function $L: \mathbb{R}^d \times \mathbb{R}^m \rightarrow \mathbb{R}$ satisfies:
\begin{enumerate}
\item There exists $\theta > 0$ such that $L(x,u) \geq \theta |u|^2$ for all $(x,u) \in \mathbb{R}^d \times \mathbb{R}^m$.
\item $L \in C^2(\mathbb{R}^d \times \mathbb{R}^m)$, and $\nabla_x L(x, u)$ has uniform linear growth in $x$: there exists $C > 0$ such that
\[|\nabla_x L(x, u)| \leq C(1 + |x|)\]
for all $(x, u) \in \mathbb{R}^d \times \mathbb{R}^m$.
\end{enumerate}
\end{assumption}
Given these assumptions we can state the following well-posedness result associtated with the PMP system.
\begin{proposition}
\label{prop:pmp-existence}
Suppose the control vector fields $f_i$ are uniformly bounded over $\mathbb{R}^d$ for $i = 1,\ldots,m$, and the drift vector field $f_0$ satisfies the linear growth condition of Assumption \ref{asmp:sublin}. Under Assumptions \ref{asmp:pmp-control} and \ref{asmp:cost}, for each $(x,p_0) \in \mathbb{R}^d \times \mathbb{R}^d$ and $T > 0$, there exists a unique solution $(\omega,p) \in H^1([0,T]; \mathbb{R}^d \times \mathbb{R}^d)$ of the Pontryagin system \eqref{eq:pmp}.

Moreover, there exists a constant $C > 0$ such that 
\[\int_I |\alpha(\omega(t),p(t))| \, dt \leq C(1 + |x| + |p_0|).\]
\end{proposition}

\begin{proof}
Under the given assumptions, the vector field associated with the Pontryagin system \eqref{eq:pmp} has linear growth and is locally Lipschitz. This ensures existence and uniqueness of solutions.

From the linear growth condition, there exist constants $M_1, M_2 > 0$ such that 
\[|\omega(t)| + |p(t)| \leq M_1 e^{M_2 t}(|x| + |p_0|).\]

The bound on the $L^1$ norm of the control follows from this estimate and the linear growth of $\alpha$ in Assumption \ref{asmp:pmp-control}:
\[\int_I |\alpha(\omega(t),p(t))| \, dt \leq M \int_I (1 + |\omega(t)| + |p(t)|) \, dt \leq C(1 + |x| + |p_0|).\]
\end{proof}


The above result establishes existence and uniqueness of solutions to the Pontryagin system. We now show that by sampling different initial conditions and initial adjoint vectors, we can construct a flow matching scheme where all interpolating trajectories are extremals of the optimal control problem.

\begin{theorem}
\label{thm:pmp-flow}
Let $\mu_0,\mu_p \in \mathcal{P}_1(\Rd)$ with $\mu_0$ compactly supported. Let $\gamma \in \mathcal{P}_1(\Rd \times \Rd)$ be given by $\gamma = \mu_0 \otimes \mu_p$. Suppose the control vector fields $f_i$ are uniformly bounded over $\mathbb{R}^d$ for $i = 1,\ldots,m$, and the drift vector field $f_0$ satisfies the linear growth condition of Assumption \ref{asmp:sublin}. Under Assumptions \ref{asmp:pmp-control} and \ref{asmp:cost}, define $S: \mathbb{R}^d \Rd \rightarrow \Omega$ by  
\[S(x_0,p_0) = (\omega,u)\] 
for all $x_0 , p_0 \in \mathbb{R}^d$, where $(\omega,p)$ solves the Pontryagin system \eqref{eq:pmp} with initial conditions $\omega(0) = x$, $p(0) = p_0$, and $u(t) = \alpha(\omega(t),p(t))$.

Then the hypotheses of Theorem \ref{thm:realization} are satisfied. Specifically, define $\mu_t := (S^{\omega}_t)_{\#}\gamma$ for $t \in I$. Then $\{\mu_t\}_{t \in I}$ solves the continuity equation \eqref{eq:continuity} with feedback control
\[u(t,x) = \int_{U} u \, d\eta_{t,x}(u)\]
where $\eta_{t,x} \in \mathcal{P}(U)$ is the conditional distribution of $U_t = S^u_t(P_0)$ given $X_t = x$, with $P_0 \sim \gamma$.

Moreover, there exists a probability measure $\mathbb{P} \in \mathcal{P}(\mathbb{R}^d \times \Gamma)$ such that 
\begin{equation}
\label{eq:reverse-flow}
\dot{\omega}(t) = -f(\omega(t), u(t,\omega(t))), \quad \omega(0) = y
\end{equation}
for $\mathbb{P}$-almost every $(y,\omega) \in \mathbb{R}^d \times \Gamma$, and $(e_t)_{\#}\mathbb{P} = \mu_t$ for all $t \in I$.
\end{theorem}

Despite the fact that the forward process is constructed from the PMP system, in general it might not be the case that forward process is optimal. However, there is a special case for which we can establish optimality.

Given this definition we prove the following sufficient condition for optimality of the constructed controller.

\begin{proposition}[Linear Convex Control]
\label{prop:linear-convex}
Assume the hypotheses of Theorem \ref{thm:pmp-flow} and $\mu_0 = \delta_{x_0}$ for some $x_0 \in \Rd$. Let $f(x,u) = Ax + Bu$ be a globally controllable linear system in the sense of Definition \ref{def:controllable}, for some matrices $A \in \mathbb{R}^{d \times d}$ and $B \in \mathbb{R}^{d \times m}$. Additionally assume $L$ is twice continuously differentiable and uniformly strongly convex with respect to the control variable: there exists $\lambda > 0$ such that 
\begin{align}
L(\alpha x + (1-\alpha)y, \alpha u + (1-\alpha)v) \leq \alpha L(x,u) + (1-\alpha)L(y,v) - \frac{\alpha(1-\alpha)\lambda}{2}|u-v|^2
\end{align}
for all $x,y \in \mathbb{R}^d$, all $u,v \in \mathbb{R}^m$ and all $\alpha \in (0,1)$. Lastly, assume that $B$ is full rank. 

Then the measure $\mathbb{P}$ guaranteed to exist in Theorem \ref{thm:pmp-flow} is unique and concentrated on solutions of the optimal control problem:
\begin{align}
\label{eq:linear-opt}
\inf_{x,u} &\int_I L(x(t),u(t)) \, dt \\
\text{subject to } &\dot{x}(t) = -Ax(t) - Bu(t) \nonumber \\
&x(0) = x_0, \quad x(T) = \omega(T) \nonumber
\end{align}
where $\omega$ is the solution of the Pontryagin system \eqref{eq:pmp}.
\end{proposition}

\begin{proof}
Under the assumptions, solutions of the Pontryagin system are locally optimal solutions of the optimal control problem \citep[24.2 Corollary]{clarke2013functional}. Since the problem is convex, the control is also globally optimal. Moreover, partial strong convexity of the cost function and linearity of the dynamics ensure that optimal trajectories and controls $(x^*, u^*) \in \Omega$ are unique due to strict convexity of the map $(x,u) \mapsto \int_I L(x(t),u(t)) \, dt$ under the constraints.

By uniqueness, optimal solutions $x^*_1, x^*_2 \in \Gamma$ for different terminal conditions $x^*_1(T) \neq x^*_2(T)$ cannot coincide at any intermediate time $t \in (0,T)$. This implies $\eta_{t,x} = \delta_{u(t,x)}$ is a Dirac measure on $\mathbb{R}^m$ for all $t \in (0,T)$ and some Borel function $u: I \times \mathbb{R}^d \rightarrow \mathbb{R}^m$. Therefore, $\mathbb{P} = (S^{\omega})_{\#}\gamma$.

To prove uniqueness of $\mathbb{P}$, suppose $\hat{\mathbb{P}}$ is another probability measure on $\Gamma$ concentrated on solutions of
\begin{equation}
\label{eq:reverse-ode}
\dot{\omega}(t) = -f(\omega(t), u(t,\omega(t))), \quad \omega(T) = y.
\end{equation}

Let $V: (0,T) \times \mathbb{R}^d \rightarrow \mathbb{R}$ be the value function. Computing:
\[\frac{d}{dt}V(t,\omega(t)) = \partial_t V(t,\omega(t)) + \langle \nabla_x V(t,\omega(t)), \dot{\omega}(t) \rangle = 0.\]

Since $L$ has positive definite Hessian with respect to $u$ (by strong convexity), the optimal control is continuous \citep[Chapter III, Corollary 6.1]{fleming2012deterministic}. The value function satisfies the Hamilton-Jacobi-Bellman PDE:
\[\partial_t V + H(x, \nabla_x V) = 0\]
where $H(x,p) = \min_{u \in \mathbb{R}^m} [-\langle p, f(x,u) \rangle + L(x,u)]$.

This implies
\[\frac{d}{dt}V(t,\omega(t)) = -L(\omega(t), u(t,\omega(t)))\]

Therefore, for any $\varepsilon \in (0,T)$:
\[V(T,\omega(T)) = V(\varepsilon,\omega(\varepsilon)) - \int_\varepsilon^T L(\omega(\tau), u(\tau,\omega(\tau))) \, d\tau.\]

Hence $(\omega, u(\cdot,\omega(\cdot)))$ restricted to $[\varepsilon,T]$ is optimal for the problem with initial condition $\omega(\varepsilon)$ and terminal condition $\omega(T)$. This implies $(\omega(t), p(t))$ satisfy the Pontryagin system \eqref{eq:pmp} for some $p_\varepsilon \in \mathbb{R}^d$.

By the growth assumption (Assumption \ref{asmp:pmp-control}), solutions are global. Therefore, $(\omega,u)$ are optimal solutions of \eqref{eq:linear-opt} over the entire interval $[0,T]$.

Given two solutions $\omega_1(t)$ and $\omega_2(t)$ of \eqref{eq:reverse-ode} with $\omega_1(T) = \omega_2(T) = y$, both satisfy the Pontryagin system with respective adjoints $p_1, p_2$. Since $B$ is full rank and the controls $u_i(t) = \arg\min_{\alpha} [-\langle p_i, f(x,\alpha) \rangle + L(x,\alpha)]$ are continuous, we have $p_1(T) = p_2(T)$. By uniqueness of solutions to the Pontryagin system (Assumptions \ref{asmp:sublin}, \ref{asmp:pmp-control}, \ref{asmp:cost}), we conclude $\omega_1 = \omega_2$.

This implies that solutions of \eqref{eq:reverse-ode} are unique for any $y \in \mathbb{R}^d$ over the time interval $[0,T]$. Next, we use this to show that the measures $\hat{\mathbb{P}}$ and $\mathbb{P}$ coincide. 

Let $e_T: \Gamma \rightarrow \mathbb{R}^d$ be the evaluation map defined by $e_T(\omega) = \omega(T)$. The measures $\mathbb{P}$ and $\hat{\mathbb{P}}$ admit disintegrations with respect to their common pushforward $(e_T)_{\#}\mathbb{P} = (e_T)_{\#}\hat{\mathbb{P}} = \mu_T$:
\[\int_{\Gamma} f(\omega) \, d\mathbb{P}(\omega) = \int_{\mathbb{R}^d} \int_{\Gamma} f(\omega) \, d\mathbb{P}_y(\omega) \, d\mu_T(y)\]
\[\int_{\Gamma} f(\omega) \, d\hat{\mathbb{P}}(\omega) = \int_{\mathbb{R}^d} \int_{\Gamma} f(\omega) \, d\hat{\mathbb{P}}_y(\omega) \, d\mu_T(y)\]
for all $f \in C_b(\Gamma)$. 

Since solutions of \eqref{eq:reverse-ode} with terminal condition $\omega(T) = y$ are unique for every $y \in \mathbb{R}^d$, both $\mathbb{P}_y$ and $\hat{\mathbb{P}}_y$ are Dirac measures concentrated on the unique trajectory reaching $y$. Therefore $\mathbb{P}_y = \hat{\mathbb{P}}_y$ for $\mu_T$-almost every $y$. Hence, the integrals with respect to any $f \in C_b(\Gamma)$ coincide, proving $\mathbb{P} = \hat{\mathbb{P}}$.
\end{proof}

Having established uniqueness of the measure $\mathbb{P}$ concentrated on optimal trajectories, we now investigate conditions under which the constructed interpolation map $S$ is regular, ensuring that the resulting flow matching scheme produces well-defined probability densities. This regularity result will be useful later to define time-reversals. Toward this end we define the following well known object in optimal control theory.

\begin{definition}[Exponential Map]
\label{def:exp-map}
The exponential map $\exp_x: I \times \mathbb{R}^d \rightarrow \mathbb{R}^d$ is defined by 
\[\exp_x(t,p_0) = \omega(t),\]
where $(\omega,p)$ is the solution of the time-reversed Pontryagin system \eqref{eq:pmp} with initial conditions $\omega(0) = x$ and $p(0) = p_0$.
\end{definition}

The regularity of the interpolation map is crucial for ensuring that the flow matching scheme produces absolutely continuous measures. We now introduce several geometric concepts that will help us characterize when the Pontryagin extremals provide a regular parametrization of the state space.

For systems with drift-free dynamics, controllability can be characterized using Lie algebraic conditions. We introduce the necessary machinery to state the Hörmander condition.

Let $\mathcal{V} = \{f_1, \ldots, f_m\}$ denote the control vector fields. The Lie bracket of two vector fields $f, g: \mathbb{R}^d \rightarrow \mathbb{R}^d$ is defined by
\begin{equation}
[f,g]_i = \sum_{j=1}^d \left(f^j \partial_{x_j} g^i - g^j \partial_{x_j} f^i\right),
\end{equation}
where $\partial_{x_j}$ denotes the partial derivative with respect to coordinate $j$.

We recursively define the Lie algebra generated by $\mathcal{V}$. Set $\mathcal{V}^0 = \mathcal{V}$. For each $k \in \mathbb{Z}_+$, define
\[\mathcal{V}^k = \{[g, h] : g \in \mathcal{V}, h \in \mathcal{V}^{j}, j = 0, \ldots, k-1\}.\]

The system satisfies the Hörmander condition (also known as the Chow-Rashevskii condition \cite{agrachev2019comprehensive}) if the Lie algebra generated by $\mathcal{V}$ spans the tangent space at every point.

\begin{assumption}[Hörmander Condition]
\label{asmp:hormander}
Suppose $f_0 \equiv 0$ and $f_i \in C^{\infty}(\mathbb{R}^d; \mathbb{R}^d)$ for each $i = 1, \ldots, m$. The Lie algebra generated by $\mathcal{V}$ has full rank: there exists $r \in \mathbb{N}$ such that
\[\mathrm{span}\{(x) : g \in \cup_{k=0}^r \mathcal{V}^k\} = \mathbb{R}^d\]
for all $x \in \mathbb{R}^d$.
\end{assumption}
The next result tells us how the support of the measures can be characterized explicitly when the system is globally controllable.
\begin{proposition}
Assume the hypothesis of Theorem \ref{thm:pmp-flow} and that support of $\mu_p$ is all of $\Rd$.  Let $L(x,u) = |u|^2$. Suppose either one of the following is true,
\begin{enumerate}
\item The system is driftless, satisfying Assumption \ref{asmp:hormander}.
\item The system is a controllable linear time invariant system.
\end{enumerate}
Then the support of $\mu_t$ is all of $\Rd$ for all $t > 0$.
\label{prop:supppmp}
\end{proposition}
\begin{proof}
For driftless systems, it is known that the range of the exponential map is dense in the state space $\Rd$ from \cite[Theorem 1]{rifford2005morse}. This along with characterization of supports of continuous maps established in Proposition \ref{prop:supp} establishes the required result.
For linear systems, it is know that they do not admit singular curves \cite[Proposition 2.4]{cannarsa2008semiconcavity}. Hence, the exponential map is in fact, all of $\Rd$.  Then the statement follows, once again using Proposition \ref{prop:supp}.
\end{proof}

For the case of linear systems, from the statement of \cite[Proposition 2.4]{cannarsa2008semiconcavity} one can see that the previous result also extends for more general kinds of costs. However, we do not state it here, it avoid stating all the additional assumptions on the costs from \cite{cannarsa2008semiconcavity}.
The usefulness of this result is that, since the exponential map propogates mass throughout the state space, one can reverse the process using Proposition \ref{prop:change-initial}. However, time-reversal is challenging due to the potential lack of regularity of the measures.
Towards, this end, to enable application of time-reversal result of Proposition \ref{prop:time-reversal}, now state the main regularity result, which ensures that the flow matching scheme based on Pontryagin extremals produces absolutely continuous measures.

\begin{proposition}[Absolute Continuity]
\label{prop:absolute-continuity}
Assume the hypothesis of Theorem \ref{thm:pmp-flow} and Proposition \ref{prop:supppmp}. Additionally,  we assume Assumption \ref{asmp:nonsingular} holds true. Let $L(x,u) = |u|^2$.  Then the family $\mu_t := (S^{\omega}_t)_{\#}\gamma$ from Theorem \ref{thm:pmp-flow} is absolutely continuous with respect to Lebesgue measure for all $t \in (0,T]$.
\end{proposition}

\begin{proof}
We first establish that the set of initial adjoint vectors $p_0 \in \mathbb{R}^d$ for which the exponential map $\exp_x(t, \cdot)$ has non-degenerate Jacobian has full Lebesgue measure. 

Let $\mathbb{S}^{d-1} \subset \mathbb{R}^d$ denote the unit sphere in $d$ dimensions, and let $\rho \in \mathcal{P}(\mathbb{S}^{d-1})$ be the uniform probability measure on $\mathbb{S}^{d-1}$. The $d$-dimensional Lebesgue measure $\lambda_d$ on $\mathbb{R}^d$ can be expressed in polar coordinates as:
\[d\lambda_d(p) = r^{d-1} \, dr \, d\rho(\theta), \quad \text{where } p = r\theta, \quad r \in [0, \infty), \quad \theta \in \mathbb{S}^{d-1}.\]

For any integrable function $f: \mathbb{R}^d \to \mathbb{R}$,
\[\int_{\mathbb{R}^d} f(p) \, d\lambda_d(p) = \int_{\mathbb{S}^{d-1}} \int_0^\infty f(r\theta) \, r^{d-1} \, dr \, d\rho(\theta).\]

To prove absolute continuity, it suffices to show that for each direction $\theta \in \mathbb{S}^{d-1}$, the set of radii $r > 0$ where the map $r \mapsto \exp_x(T, r\theta)$ has degenerate differential is at most countable. Under the Hörmander condition (Assumption \ref{asmp:hormander}) and the non-singularity assumption (Assumption \ref{asmp:nonsingular}), this follows from \cite[Corollary 8.51]{agrachev2019comprehensive}, which states that conjugate points along extremal trajectories form a discrete set.

Since the set of singular points has measure zero in each radial direction, the set of $p_0 \in \mathbb{R}^d$ where $D_{p_0}\exp_x(T, p_0)$ is singular has Lebesgue measure zero. By \cite[Theorem 9.2.2]{bogachev2010differentiable}, if $\gamma$ is absolutely continuous with respect to Lebesgue measure and the map $p_0 \mapsto \exp_x(T, p_0)$ has non-degenerate Jacobian $\gamma$-almost everywhere, and $ ( \exp_x( T, \cdot)_{\#}\mu_p$ is absolutely continuous.
Now, 
\begin{equation}
\label{eq:kern}
\mu_T(A) = \int_{\mathbb{R}^d} (\exp_x(T, \cdot))_{\#}\mu_p(A) \, d\mu_0(x)
\end{equation}
for all Borel measurable sets $A \subseteq \Rd$.
It is clear from this expression that $\exp_x( T, \cdot)_{\#}\mu_p$ being absolutely continuous, gives absolute continuity of $\mu_T$ due to this expression.

The same argument applies for any $t \in (0,T]$, establishing absolute continuity of $\mu_t$ for all positive times. 
\end{proof}

Given these results, we can justify the time reversal algorithm and establish absolute continuity of the flow matching scheme as follows.

From Proposition \ref{prop:absolute-continuity}, we know that $\mu_T = (S^\omega_T)_{\#}\gamma$ is absolutely continuous with respect to Lebesgue measure when $\mu_p$ is absolutely continuous. Moreover, when $\mu_p$ has full support, so does $\mu_T$.

Next, the key insight is that this absolute continuity allows us to change from $\mu_T$ to another absolutely continuous measure $\tilde{\mu}_T$ that is easy to sample from (such as a Gaussian distribution), and Proposition \ref{prop:change-initial} ensures that the behavior of the resulting dynamics remains the same. 

The construction of flows for this case immediately follows.

\begin{enumerate}
\item Pick $\mu_0$ that is supported on a target set that is required to be stabilized.
\item Replace the potentially unknown measure $\mu_T$ with any convenient alternative absolutely continuous measure $\tilde{\mu}_T$ (e.g., $\mathcal{N}(0, I)$).
\item By Proposition \ref{prop:change-initial}, the trajectories under $\tilde{\mu}_T$ still solve the same control system dynamics.
\item By Proposition \ref{prop:change-initial}, the time-reversed system will be transferred to the initial state value for the PMP system. 
\end{enumerate}

This is particularly powerful for implementation in the sense that rather than needing to sample from the exact distribution $\mu_T$ (which depends on the adjoint dynamics and may be complex), we can sample from any convenient absolutely continuous distribution and still obtain valid flow matching trajectories. The dynamics automatically adjust through the disintegration to ensure correct transport and hence spreading over $\Omega$, though the nature of the distribution cannot be controlled.

\subsection{Variable End-Point Problems}

We wish to mimic the developments of the previous section to variable end point problems. In this case, the time-reversed PMP system we get is the following,

\begin{align}
\label{eq:pmpve}
\dot{\omega}(t) &= -f(\omega(t), \alpha(\omega(t), p(t))) \\
\dot{p}(t) &= \langle \nabla_x f(\omega(t), \alpha(\omega(t), p(t))), p(t) \rangle + \nabla_x L(\omega(t), \alpha(\omega(t), p(t))) \nonumber \\
\omega(0) &= x_0 \nonumber \\
p(0) &= \nabla \Psi (x_0) \nonumber  \\
\alpha(\omega, p) &= \arg\min_{\alpha \in \mathbb{R}^m} \left[-\langle p, f(\omega, \alpha) \rangle + L(\omega, \alpha)\right] \nonumber
\end{align}

This system corresponds to necessary conditions of optimality for the optimal control problem,
\begin{align}
\label{eq:nonlinear-opt}
\inf_{x,u} & \int_I L(x(t),u(t)) \, dt + \Psi (x(T))  \\
\text{subject to} ~ & \dot{x}(t) = f(x,u) \nonumber \\
& x(0) = \omega(T)  \nonumber
\end{align}
We will need an additional assumption due to the end-point cost function $\Psi$.
\begin{assumption}
\label{asmp:termcost}
The function $\Psi \in C^2(\mathbb{R}^d)$, and $\nabla_x \Psi $ has uniform linear growth: there exists $C > 0$ such that
\[|\nabla_x \Psi(x)| \leq C(1 + |x|)\]
for all $x \in \mathbb{R}^d$.
\end{assumption}

Given these assumptions we state the following results on existence of superpositions, and recovering of the solution to the optimal control problem for the linear convex case. Since the proofs follow nearly verbatim as in the proofs of the previous section, we skip the proofs, for all but the last result on the support of the constructed measures $\mu_t$. 

\begin{theorem}
\label{thm:pmp-flow2}
Let $\mu_0 \in \mathcal{P}_1(\Rd)$. Let $\gamma \in \mathcal{P}_1(\Rd )$ be given by $\gamma = \mu_0$. Suppose the control vector fields $f_i$ are uniformly bounded over $\mathbb{R}^d$ for $i = 1,\ldots,m$, and the drift vector field $f_0$ satisfies the linear growth condition of Assumption \ref{asmp:sublin}. Under Assumptions \ref{asmp:pmp-control},  \ref{asmp:cost} and \ref{asmp:termcost}, define $S: \mathbb{R}^d \Rd \rightarrow \Omega$ by  
\[S(x_0) = (\omega,u)\] 
for all $x_0 \in \mathbb{R}^d$, where $(\omega,p)$ solves the Pontryagin system \eqref{eq:pmpve} with initial conditions $\omega(0) = x$, and $u(t) = \alpha(\omega(t),p(t))$.

Then the hypotheses of Theorem \ref{thm:realization} are satisfied. Specifically, define $\mu_t := (S^{\omega}_t)_{\#}\gamma$ for $t \in I$. Then $\{\mu_t\}_{t \in I}$ solves the continuity equation \eqref{eq:continuity} with feedback control
\[u(t,x) = \int_{U} u \, d\eta_{t,x}(u)\]
where $\eta_{t,x} \in \mathcal{P}(U)$ is the conditional distribution of $U_t = S^u_t(P_0)$ given $X_t = x$, with $P_0 \sim \gamma$.

Moreover, there exists a probability measure $\mathbb{P} \in \mathcal{P}(\mathbb{R}^d \times \Gamma)$ such that 
\begin{equation}
\label{eq:reverse-flow}
\dot{\omega}(t) = -f(\omega(t), u(t,\omega(t))), \quad \omega(0) = y
\end{equation}
for $\mathbb{P}$-almost every $(y,\omega) \in \mathbb{R}^d \times \Gamma$, and $(e_t)_{\#}\mathbb{P} = \mu_t$ for all $t \in I$.
\end{theorem}

For linear problems with convex costs, once again, we can state a stronger result.

\begin{proposition}[Linear Convex Control]
\label{prop:linear-convex2}
Assume the hypotheses of Theorem \ref{thm:pmp-flow2} and $\mu_0 = \delta_{x_0}$ for some $x_0 \in \Rd$. Let $f(x,u) = Ax + Bu$ be a globally controllable linear system in the sense of Definition \ref{def:controllable}, for some matrices $A \in \mathbb{R}^{d \times d}$ and $B \in \mathbb{R}^{d \times m}$. Additionally assume $L$ is twice continuously differentiable and uniformly strongly convex with respect to the control variable: there exists $\lambda > 0$ such that 
\begin{align}
L(\alpha x + (1-\alpha)y, \alpha u + (1-\alpha)v) \leq \alpha L(x,u) + (1-\alpha)L(y,v) - \frac{\alpha(1-\alpha)\lambda}{2}|u-v|^2
\end{align}
for all $x,y \in \mathbb{R}^d$, all $u,v \in \mathbb{R}^m$ and all $\alpha \in (0,1)$. Similarly, assume that $\Psi$ is strictly convex. Lastly, assume that $B$ is full rank. 

Then the measure $\mathbb{P}$ guaranteed to exist in Theorem \ref{thm:pmp-flow2} is unique and concentrated on solutions of the optimal control problem:
\begin{align}
\label{eq:linear-opt}
\inf_{x,u} &\int_I L(x(t),u(t)) \, dt + \Psi(x(T)) \\
\text{subject to }  &\dot{x}(t) = -Ax(t) - Bu(t)  \\
&x(0) = x_0
\end{align}
\end{proposition}

\begin{algorithm}
\caption{Flow Matching Along Pontryagin System through Time Reversal (Variable-end Point Problems)}
\label{alg:pmp-flowve}
\begin{algorithmic}[1]
\State Sample initial condition $x  \sim \mu_0$. 
\State Define the interpolation path $X_t$ and corresponding control $U_t$ using the map $S: \mathbb{R}^d \rightarrow \Omega$ defined by solving \eqref{eq:pmpve}.
\State Learn control law $u: I \times \mathbb{R}^d \rightarrow \mathbb{R}^m$ by solving the regression problem:
\[\min_{u} \int_I \mathbb{E}\left[\|U_t - u(t, X_t)\|^2\right] dt\]
\State Sample new data via learned flow. Sample $Z_T \sim \mu_T$ for some $\mu_T \in \mathcal{P}(\mathbb{R}^d)$ and solve the ODE in reverse:
\[\dot{Z}_t = f(Z_t, u(t, Z_t)) = f_0(Z_t) + \sum_{i=1}^m u_i(t, Z_t) f_i(Z_t)\]
or equivalently, sample $\tilde{Z}_0 \sim \mu_0$ and solve forward:
\[\dot{\tilde{Z}}_t = -f(\tilde{Z}_t, u(T-t, \tilde{Z}_t)) = -f_0(\tilde{Z}_t) - \sum_{i=1}^m u_i(T-t, \tilde{Z}_t) f_i(\tilde{Z}_t)\]
\State \textbf{Output:} $Z_0$ (from reverse integration) or $\tilde{Z}_T$ (from forward integration)
\end{algorithmic}
\end{algorithm}

In the next proposition, we characterize the support of the measure $\mu_t$. Note that we don't need any assumption on the surjectivity of the end-point map.

\begin{proposition}
Assume the hypothesis of Theorem \ref{thm:pmp-flow2} and that support of $\mu_0$ is all of $\Rd$.  

Then the support of $\mu_t$ is all of $\Rd$ for all $t > 0$.
\end{proposition}
\begin{proof}
Given the assumptions, for every $x_0 \in \Rd$, there exists an optimal control for the optimal control problem \cite{cannarsa2008semiconcavity}[Theorem 7.4.5],
\begin{align}
\label{eq:nonlinear-optsec}
\inf_{x,u} & \int_I L(x(t),u(t)) \, dt + \Psi (x(T)) \\
\text{subject to} ~ & \dot{x}(t) = f(x,u) \\
& x(0) = x_0
\end{align}
Moreover, the solution is given by the PMP system \cite{cannarsa2008semiconcavity}[Theorem 7.4.17]. Hence, the map from $\omega(T)$ to $\omega(0)$ is surjective, where $\omega$ is the solution of the PMP system \eqref{eq:pmpve}. Then the result follows from Proposition \ref{prop:supp}.

\end{proof}

To summarize the developments of this section, we have developed a control-theoretic version of noising and denoising methodology by leveraging the Hamiltonian system associated with PMP. We believe this is a novel alternative to stochastic differential equation based approaches to noising the system, such as developed in \cite{song2020score}, and might of independent interest even in the standard generative modeling setting, that does not involve control-dynamical constraints on the trajectories of the system.

\section{Numerical Examples}
\label{sec:numexp}
In this section we consider several examples for demonstrating the different algorithms presented in the appear. In each of the examples, the controller was learned using a simple multilayer perceptron, with depth varying from $1$ to $3$, from problem to problem.
\subsection{Minimum Energy Measure Interpolants for Linear Systems}
First, we mention the case of achieving measure-to-measure interpolation of linear control systems.

Consider the linear system
\begin{align}
\label{eq:linsys}
& \dot{x}(t) = Ax(t) +B u(t) \nonumber \\
& x(0) = x_0
\end{align}
where $A \in \R^{d \times d}$ and $B \in \R^{d \times m}$. We assume that 
$(A,B):=(A,[b_1,...,b_m])$ satisfies the Kalman rank condition, 
\begin{equation}
{\rm rank}~ [B~AB~....A^{d-1}B] = d
\end{equation}
The minimum energy control between two points $x_0$ and $x_1$ can be constructed in closded form.
We define the controllability grammian
$W  = \int_0^1 e^{At}BB^Te^{A^Tt}dt$. If the system is {\it controllable}, that is, any two points in space can be connected by an admissible trajectory, then the grammian is invertible.  Using this matrix we can construct the interpolating map. 
The optimal minimum energy control connecting $x$ and $y$ is given by.

\[S^u_t(x_0,x_T) =B^T e^{A^t(1-t)} W^{-1}(y - e^{A}z),~~~ t\in I \]
Then the interpolating map is given by

\[S^{\omega}_t(x_0,x_T) = e^{At}x_0 + \int_0^t e^{A(t-\tau)}Bu^{x,y}(\tau)d\tau \]
This setting is identical to the one considered in \cite{mei2025flow}, except we do not alow for noise. 
\subsection{Approximate Output Flow Matching using Feedback Control of Linear Systems}
 The Grammian can be computationally challenging to compute. For this reason, we provide an alternative approach to construct interpolants.
 Once again assume that the system is a linear controllable system according to the last section. Suppose $\mu_f$ is supported on the set of points $E_{eq} = \{ y \in \Rd; \exists u ~ \text{s.t}~ Ay+Bu = 0\}$. These are the set of points at which the system \eqref{eq:linsys} can be at equilibrium. Since the system is controllable, it is also stabilizable. Hence, there exists a matrix $K \in \R^{m \times d} $ such that there exists $\alpha^{y} \in \R^m$, such that for $u(t) = K(\omega - y) + u^y$ the closed-loop system

\begin{align}
\label{eq:clplinsys}
& \dot{\omega}(t) = A\omega(t)+B  K(\omega(t) - y) + B\alpha^y  \nonumber \\
& \omega(0) = x
\end{align}

satisfies
\[\|\omega(t) - y\| \leq M e^{\lambda t } \|x - y\| \]
Therefore, if $t$ is large enough we can ensure $\|\omega(t) - y\| \leq \varepsilon $, facilitating approximate flow matching.

There are several ways to compute this matrix $K$ in the control theory literature. Note that $K$ is independent of $x$ and $y$ and only dependent on $A$ and $B$. An advantage of this approach over the minimum energy approach is that its easier to compute a stabilizing $K$ in comparison to computing the controllability grammian $W$. In the following example, we consider a six-dimensional system with two dimensional equilibrium set. The interpolating trajectories are constructed by using a pole placement based controller that steers the system from the initial to the final states.



\begin{figure}[!ht]
    \centering
    \begin{subfigure}{0.4\linewidth}
        \centering
        \includegraphics[width=\linewidth]{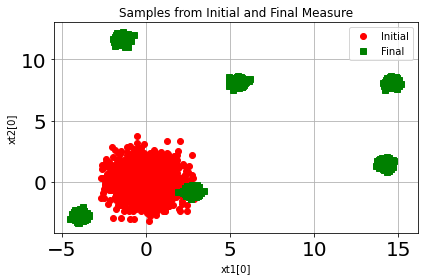}
        \caption{Initial and Target output positions of a six-state linear system.}
        \label{fig:poleplacetarg}
    \end{subfigure}

    \vspace{0.5cm} 

    \begin{subfigure}{0.4\linewidth}
        \centering
        \includegraphics[width=\linewidth]{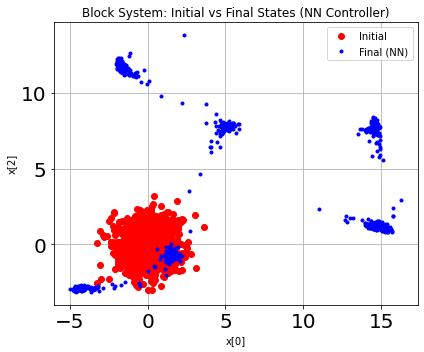}
        \caption{Initial and trained output positions of a six-state linear system.}
        \label{fig:poleplace}
    \end{subfigure}

    \caption{Comparison of final output positions projections of target distribution and trained six-state linear system.}
    \label{fig:LSCC}
\end{figure}

\FloatBarrier
\subsection{Measure interpolation for nonlinear driftless systems}
Controlling nonlinear systems is significantly harder than linear systems. Inspite of this, there do exist a number of methods tailored for constructing interpolating paths between initial and terminal points \cite{jean2014control}. Consider for example the three dimensional system with $2$ control inputs. This system is known to be controllable.

\begin{align}
\dot{x}_1(t) = u_1(t) ~~ \dot{x}_2(t) = u_2(t) ~~ \dot{x}_3(t) = u_1(t)x_2(t) 
\label{eq:drsys1}
\end{align}

An algorithm to steer systems of this form is the following as shown in \cite{jean2014control}. Move $(x_1,x_2)$ to $(y_1,y_2)$ using constant controls
by setting controls 
\[(u_1(t),u_2(t)) = (y_1 \frac{t}{2 \pi}  + \frac{1-t}{2 \pi} x_1,y_2\frac{t}{2 \pi} + \frac{1-t}{2 \pi} x_2)  ~~ t \in (0,2 \pi).\] 
Next, one selects sinusoidal signals of appropriate frequencies to transfer the third coordinate. For instance, one choice of controls is

\begin{align*}(u_1(t),u_2(t)) =(  \sin t,  \frac{y_3 - \omega_3(2 \pi)}{\pi} \cos t ),  \\
t\in (2 \pi, 4 \pi] 
\end{align*}
Then it can be shown that $\omega(0) = x$ and $\omega(4 \pi) = y$. Partially, this is due to the fact that the controls in the second phase integrate to zero over the interval $(2 \pi, 4 \pi] $ and leave the first two coordinates unchanged, while transferring the third coordinate to the required position.

\begin{figure}[!ht]
    \centering
    \begin{subfigure}{0.4\linewidth}
        \centering
        \includegraphics[width=\linewidth]{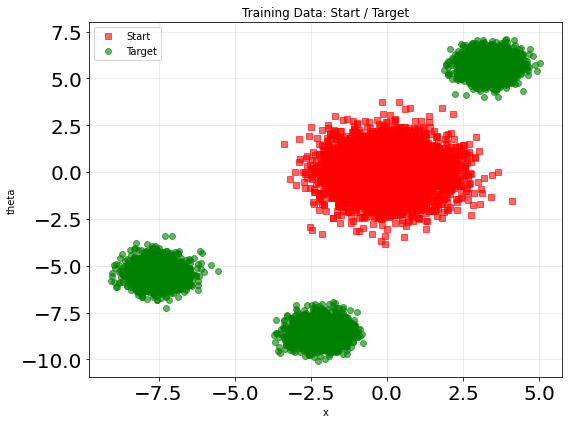}
        \caption{Initial and Target positions. Projection of Target Distribution of a 3D driftless system.}
        \label{fig:noentropy_a}
    \end{subfigure}
 \hfill
    \begin{subfigure}{0.4\linewidth}
        \centering
        \includegraphics[width=\linewidth]{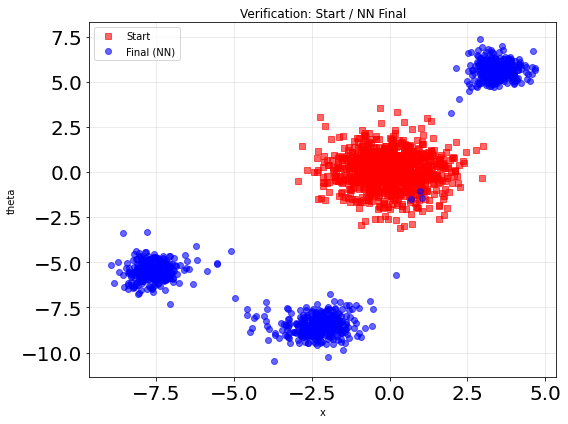}
        \caption{Trained final positions Projection of Target Distribution of a driftless nonlinear system}
        \label{fig:controlaffine}
    \end{subfigure}

    \FloatBarrier

    \caption{Comparison of final positions projections of target distribution and trained driftless system \eqref{eq:drsys1}.}
    \label{fig:noentropy_combined}
\end{figure}

\subsection{Randomized control-based Flow Matching for Stabilization}

In this section, we show a numerical experiment on the Martinet system,
\begin{align}
& \dot{x} = u_1, \\
& \dot{y} = u_2, \\
& \dot{z} = \tfrac{1}{2} y^2 u_1,
\end{align}
where the controls $(u_1,u_2)$ are sampled as continuous Brownian motions initialized at zero.  
The objective is to stabilize the system to the origin by learning a reverse-time policy.  This example violates the assumption on the absence of singular minimizers, as abnormal extremals are known to exist. Perhaps due to this missing regularity, the trajectories do not exactly converge to the origin, but to some one-dimensional curve.
The results of the time-reversal for multiple initial conditions are shown in Figures~\ref{fig:martinet}. 

\begin{figure}[!ht]
    \centering
    \includegraphics[width=0.6\linewidth]{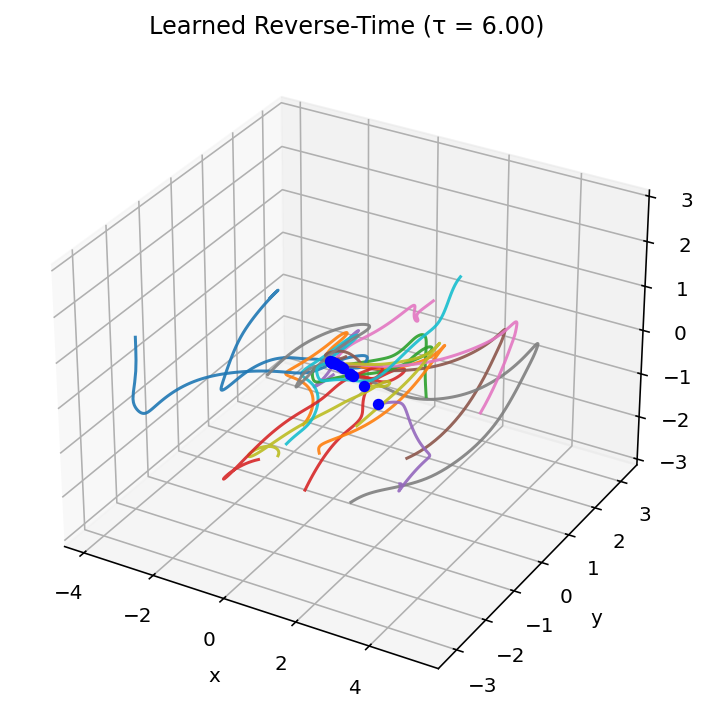}
    \caption{Time-reversed trajectories of the Martinet system visualized in 3D. 
    The learned policy drives a wide range of initial conditions close to the origin.}
    \label{fig:martinet}
\end{figure}

\subsection{PMP-based Flow Matching for Stabilization}
In this section, we show a numerical experiment to stabilize the unicycle model,

\begin{align}
& \dot{x} = v \cos \theta \\ 
& \dot{y} = v \sin \theta \\
& \dot{\theta} = u
\end{align}
The first objective is to stabilize the system to the origin. The results of the time-reversal for multiple initial conditions can be seen in Figures \ref{fig:unicycle}. The second objective is to stabilize the system to the unit sphere. This is shown for a larger number of conditions in Figure \ref{fig:unicycle2}

\begin{figure}[!ht]
    \centering
    \includegraphics[width=0.5\linewidth]{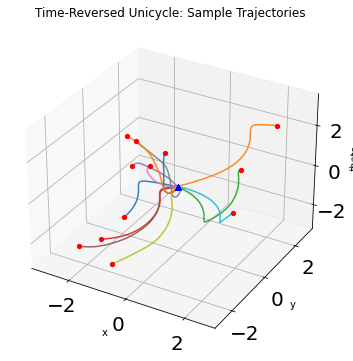}
    \caption{Trajectories of the time-reverse system visualized in 3D, stabilized to the origin.}
    \label{fig:unicycle}
\end{figure}

\begin{figure}[!ht]
    \centering
    \includegraphics[width=0.5\linewidth]{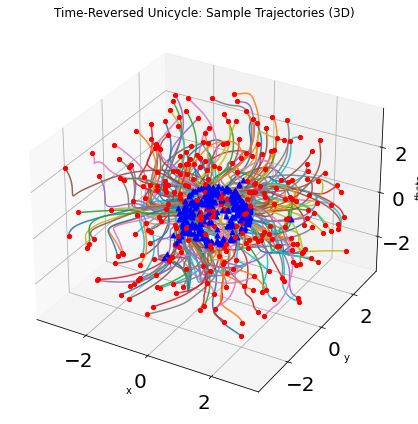}
    \caption{Trajectories of the time-reverse system visualized in 3D, stabilized to the unit sphere.}
    \label{fig:unicycle2}
\end{figure}

\subsection{Flow Matching for Path Planning}
\label{sec:rrtstar}

In this section, we show how flow matching can be used with the principle of denoising to path planning, by using popular path planners that construct open-loop trajectories to synthesize a feedback law. RRT* \cite{karaman2011sampling}  is a well-known algorithm to construct path planning trajectories.

Given an initial configuration $x_0$, we run RRT* {\it without a specified goal} to grow an exploration tree rooted at $x_0$. This produces a collection of $N$ open-loop trajectories $\{x_i(\cdot)\}_{i=1}^N$, each connecting $x_0$ to a distinct target $x_i^f$ sampled from the tree. To construct a stabilizing feedback law, we reverse these trajectories in time so that each runs from $x_i^f$ to $x_0$, reparameterize them by arc length over $t \in [0,1]$, and compute the corresponding open-loop velocities $v_i(t) = \dot{x}_i(t)$.  We then learn a feedback controller $u_\theta(x, t)$, using the flow matching loss. See
\ref{alg:rrt-flow-matching} for the pseudoalgorithm.

Since the learned vector field may point into obstacle regions due to sparse data, we use a heuristic to integrate the closed-loop system $\dot{x} = u_\theta(x, t)$ as a projected dynamical system \cite{nagurney2012projected}, projecting the state onto the boundary of the free space $\mathcal{X}_{\text{free}}$ at each integration step. Additionally, noise is added at the boundary to prevent the trajectories in the reverse direction from being stuck in at the boundary. 
The results of a numerical experiment can be seen in Figure \ref{alg:rrt-flow-matching}.

\begin{figure}[!ht]
    \centering
    \includegraphics[width=0.5\linewidth]{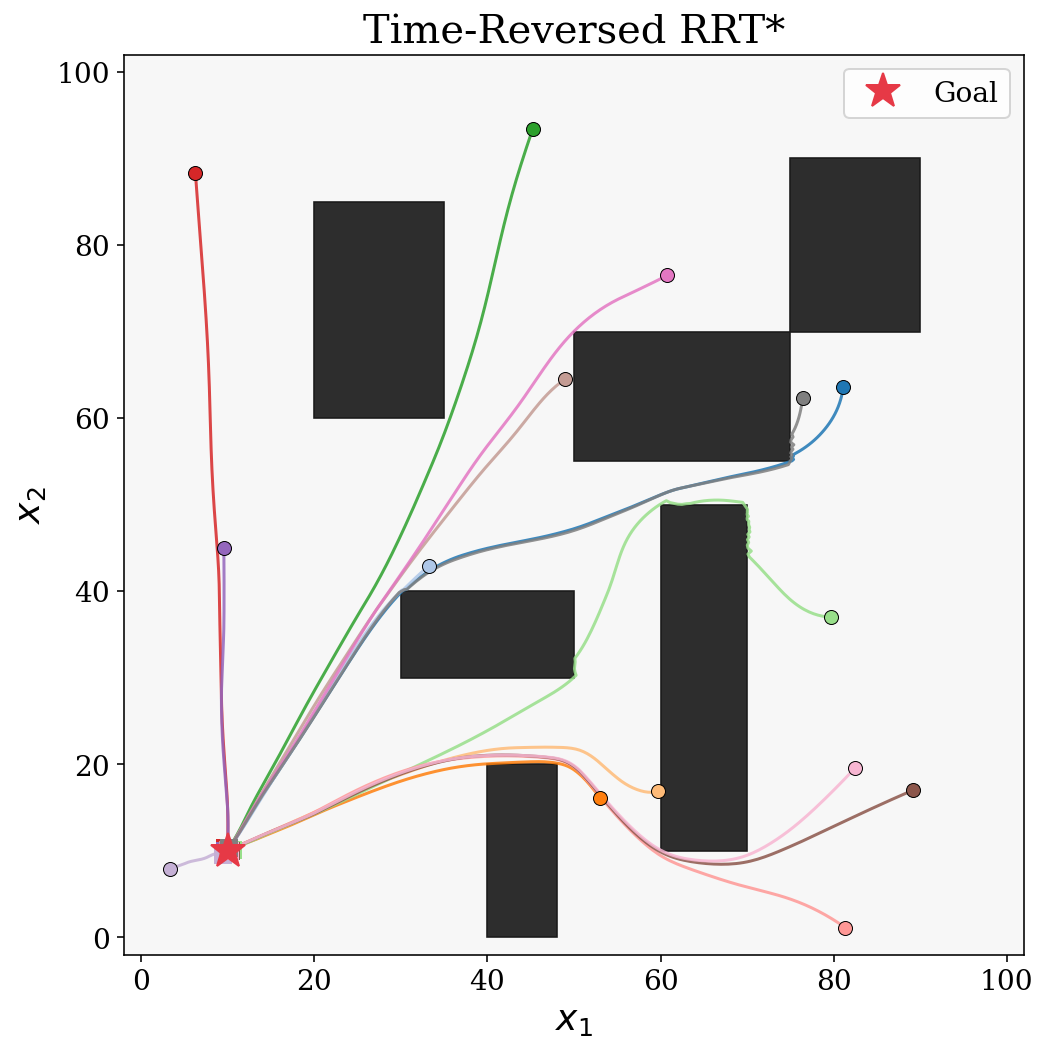}
    \caption{Trajectories of the time-reversed system using RRT* based noising in a domain with rectangular obstacles.}
    \label{fig:rrts1}
\end{figure}

\begin{algorithm}
\caption{Flow Matching for Path Planning via RRT* and Time Reversal}
\label{alg:rrt-flow-matching}
\begin{algorithmic}[1]
\State Fix target configuration $x_0 \in \mathcal{X}_{\mathrm{free}}$. Run RRT* from $x_0$ without a specified goal for $K$ iterations, growing an exploration tree $\mathcal{T}$ rooted at $x_0$ in the free space $\mathcal{X}_{\mathrm{free}} = \mathbb{R}^d \setminus \mathcal{O}$.
\State Sample $N$ nodes $\{x_i^f\}_{i=1}^N$ from $\mathcal{T}$. For each node, extract the trajectory connecting $x_0$ to $x_i^f$, reverse it in time, and reparameterize by arc length over $t \in [0,1]$. 
\State Define the interpolation path $X_t = \tilde{x}_i(t)$ and corresponding velocity $V_t = \dot{\tilde{x}}_i(t)$, where $\tilde{x}_i(0) = x_i^f$ and $\tilde{x}_i(1) = x_0$.
\State Learn control law $u: I \times \mathbb{R}^d \rightarrow \mathbb{R}^d$ by solving the regression problem:
\[\min_{u} \int_I \mathbb{E}\left[\|V_t - u(t, X_t)\|^2\right] dt\]
\State Sample new data via learned flow. Sample $Z_0 \sim \mu_0$ for some $\mu_0 \in \mathcal{P}(\mathcal{X}_{\mathrm{free}})$ and integrate the projected dynamical system:
\[\dot{Z}_t = \Pi_{T_{Z_t}\mathcal{X}_{\mathrm{free}}}(u(t, Z_t))\]
where $\Pi_{T_{Z_t}\mathcal{X}_{\mathrm{free}}}$ denotes projection onto the set of admissible velocities at any point.
\State \textbf{Output:} $Z_1 \approx x_0$.
\end{algorithmic}
\end{algorithm}

\section{Conclusion and Future Outlook} 
In summary, we have extended the framework of flow matching to the setting of controlled dynamical systems, enlarging its scope beyond the usual interpolation between measures. By introducing latent variables into the interpolating maps, we obtain variants such as output flow matching and approximate flow matching, thereby linking classical control techniques—feedback, stabilization, and oscillation-based planning—to measure transport. At the same time, the method can be inverted to address stabilization itself, with the target measure taken as a Dirac mass or a set-supported distribution. Two constructions are proposed in this vein: one derived from the adjoint equations of the Maximum Principle, the other from randomized controls reminiscent of diffusion models. Together these yield four distinct algorithms, each accompanied by existence and regularity results, and tested numerically on illustrative examples. The results indicate both the promise of flow matching as a method for control design. 

We note that there are certain \textbf{open questions} that the work raises. For instance, while we showed existence of solutions of the constructed flows, it is not clear if uniqueness of solutions of flows can be guaranteed in any setting. This is a significant shortcoming, as to conclusively establish the suitability of the method for control physical applications, without uniqueness, no such guarantee can be given. The answer to uniqueness is negative in even simple settings of unconstrained flow matching. Such as if the initial and final measures are sum of two Diracs and one uses two intersecting lines to interpolate measures.

\section{Acknowledgment}

The author thanks Ludovic Rifford and Emmanuel Trélat for their valuable insights and discussions on the properties of the exponential map, which significantly enhanced the quality of this paper.

\FloatBarrier
\appendix

\section{Appendix}
This section collects some supplementary results that are used in the main body of the paper.



\begin{proposition}\textbf{(Time reversal of superposition solutions)}
\label{prop:time-reversal}
Suppose $v: I \times \Rd \rightarrow \Rd$ is a Borel vector field. Let $I \ni t \mapsto \mu_t \in \mathcal{P}(\Rd)$ be such that $\mu_t$ is a weak solution to the continuity equation according to 
\begin{equation}
\partial_t \mu+  \nabla \cdot \left(v(t,x)) \mu \right) =0, ~ \text{in}~ \Rd \times I 
\end{equation}
and 
\[ \int_0^T \int_{\Rd} \frac{|v(t,x)|}{1+|x|}d\mu_t(x) dt < \infty \]
Let $\mathbb{P} \in \mathcal{P}(C(I;\Rd))$ be the probability measure of the superposition principle from Theorem \ref{thm:superposition}. Then $\mu^{\rm rev}_t : = \mu_{t-T}$ is a weak solution to the continuity equation 
\begin{equation}
\partial_t \mu^{\rm rev}+  \nabla \cdot \left(v(t-T,x)) \mu^{\rm rev}  \right) =0, ~ \text{in}~ \Rd \times I 
\end{equation}
Moreover, $\mathbb{P}^{rev}$ is concentrated on the solutions of the ordinary differential equation.
\begin{equation}
	\dot{\omega}(t) = v(t-T,\omega(t)),~~ \gamma(0) = y
\end{equation}
for $\mathbb{P}^{rev}$ almost every $\omega \in C(I;\Rd)$, and $\mathbb{P}(\omega_t \in dx) = d\mu^{\rm rev}_t(x)$, for all $t \in I$.
\end{proposition}

\begin{proposition}[Change of Initial Measure and Support of Final Measure]
\label{prop:change-initial}
Let $v: I \times \mathbb{R}^d \rightarrow \mathbb{R}^d$ be a Borel measurable vector field. Let $\mathbb{P} \in \mathcal{P}(\Gamma)$ be a probability measure such that for $\mathbb{P}$-almost every $\omega \in \Gamma$, we have
\begin{equation}
\label{eq:ode-v}
\dot{\omega}(t) = v(t,\omega(t)), \quad t \in I.
\end{equation}

Let $\mu_0 := (e_0)_{\#}\mathbb{P}$ and let $\{\mathbb{P}_x\}_{x \in \mathbb{R}^d}$ be the disintegration of $\mathbb{P}$ with respect to $\mu_0$. Suppose $\tilde{\mu}_0 \in \mathcal{P}(\mathbb{R}^d)$ is absolutely continuous with respect to $\mu_0$. Define the measure $\mathbb{P}^{\mathrm{new}} \in \mathcal{P}(\Gamma)$ by
\[\int_{\Gamma} g(\omega) \, d\mathbb{P}^{\mathrm{new}}(\omega) = \int_{\mathbb{R}^d} \int_{\Gamma} g(\omega) \, d\mathbb{P}_x(\omega) \, d\tilde{\mu}_0(x)\]
for all $g \in C_b(\Gamma)$.

Then $\mathbb{P}^{\mathrm{new}}$-almost every $\omega \in \Gamma$ solves the ODE \eqref{eq:ode-v}. Particularly, $\mu_t : = (e_t)_{\#}\mathbb{P}$ solves the continuity equation \eqref{eq:ctnteq}.

Moreover, ${\rm supp} ~\tilde{\mu}_T ~ \subseteq {\rm supp} ~ \mu_T$.

In the special case, if $\mu_T  = (e_T)_{\#}\mathbb{P} = \delta_{x_T}$ for some $x_T \in \mathbb{R}^d$, then $\tilde{\mu}_T := (e_T)_{\#}\mathbb{P}^{\mathrm{new}} = \delta_{x_T}$. 

If $\mu_0$ and $\tilde{\mu}_0$ are mutually absolutely continuous. Then  ${\rm supp} ~\tilde{\mu}_T  = {\rm supp} ~ \mu_T$.
\end{proposition}


\begin{proof}
By the disintegration theorem, we have
\[\int_{\Gamma} g(\omega) \, d\mathbb{P}(\omega) = \int_{\mathbb{R}^d} \int_{\Gamma} g(\omega) \, d\mathbb{P}_x(\omega) \, d\mu_0(x)\]
for all $g \in C_b(\Gamma)$.

Let $h \in L^1(\mu_0)$ be the Radon-Nikodym derivative of $\tilde{\mu}_0$ with respect to $\mu_0$, so that $d\tilde{\mu}_0(x) = h(x) \, d\mu_0(x)$. Since the map
\[x \mapsto \int_{\Gamma} g(\omega) \, d\mathbb{P}_x(\omega)\]
is bounded and measurable for each $g \in C_b(\Gamma)$, the product
\[x \mapsto h(x) \int_{\Gamma} g(\omega) \, d\mathbb{P}_x(\omega)\]
is in $L^1(\mu_0)$. Therefore, the measure $\mathbb{P}^{\mathrm{new}}$ can be equivalently expressed as
\[\int_{\Gamma} g(\omega) \, d\mathbb{P}^{\mathrm{new}}(\omega) = \int_{\mathbb{R}^d} \left(\int_{\Gamma} g(\omega) \, d\mathbb{P}_x(\omega)\right) h(x) \, d\mu_0(x),\]
which is well-defined.
If $A \subset \Gamma$ is a null set for $\mathbb{P}$. Then $\mathbb{P}_x(A)$ is $0$ for $\mu_0$ almost every $x$. This implies that $\mathbb{P}^{\mathrm{new}}(A) = 0 $.
Hence, $\mathbb{P}^{\mathrm{new}}$ gives zero measure to the null sets of $\mathbb{P}$. From this we conclude that $\mathbb{P}^{\mathrm{new}}$-almost every trajectory solves \eqref{eq:ode-v}.

From the previous argument, it is clear that $\mathbb{P}^{\mathrm{new}}$ is absolutely continuous with respect to the measure $\mathbb{P}$. The property of absolute continuity is preserved under pushforward. Hence, $\tilde{\mu}_T$ is absolutely continuous with respect to $\mu_T$. This implies the support inclusion:  ${\rm supp} ~\tilde{\mu}_T ~ \subseteq {\rm supp} ~ \mu_T$.

The same argument extends to the mutually absolutely continuous case, since $h>0$ for  $\mu_0$ almost everywhere on $\Rd$, and hence $\mathbb{P}$ and $\mathbb{P}^{\mathrm{new}}$ are mutually absolutely continuous.
\end{proof}

\begin{proposition}
\label{prop:supp}
Let $X$ and $Y$ be topological spaces, and $F: X \rightarrow Y$ is a continuous function. Let $\mu \in \mathcal{P}(X)$. Then  

\begin{equation}
{\rm  supp} ~ F_{\#} \mu = {\rm cl}~F({\rm  supp} \mu)
\end{equation}
\end{proposition}
\begin{proof}
Suppose  $y_1 \in {\rm cl}~F({\rm  supp} \mu)$ with neighborhood $N_{y_1}$. Clearly, $N_{y_1} \cap {\rm  supp}$ is non-empty and there exists a point $y_2 \in F({\rm  supp} ~\mu)$. Hence, $N_{y_1}$ is a neighborhood of $y_2$. Let $x_2 \in {\rm  supp} \mu$ be such that $F(x_2) = y_2$. From continuity of $F$, $F^{-1} (N_{y_1})$ is a neighborhood of the point $x_2$. Therefore,
\[
F_{\#}\mu(N_{y_1}) = \mu\!\left(F^{-1}(N_{y_1})\right) > 0
\]
Therefore ${\rm cl}~F({\rm  supp}~ \mu) \subseteq {\rm  supp} ~ F_{\#} \mu$. 

For the reverse inclusion, suppose $y_1 \notin \overline{F({\rm supp}\,\mu)}$. 
Then there exists an open neighborhood $N_{y_1}$ of $y_1$ such that $
N_{y_1} \cap F({\rm supp}\,\mu)$ is empty.
Equivalently, $F^{-1}(N_{y_1}) \cap {\rm supp}\,\mu$ is empty. 
Since $F^{-1}(N_{y_1})$ is open, this implies 
\[
F_{\#}\mu(N_{y_1}) = \mu(F^{-1}(N_{y_1})) = 0.
\]
which means $y_1 \notin {\rm supp}\,F_{\#}\mu$. 
Therefore
\[
{\rm supp}\,F_{\#}\mu \subseteq \overline{F({\rm supp}\,\mu)} .
\]
\end{proof}
\begin{proposition}
    Suppose $f_i :\mathbb{R}^d \rightarrow \mathbb{R}^d$ are smooth, globally bounded vector-fields for $i =1,...,m$, and the cost function $L: \mathbb{R}^d \times \mathbb{R}^m \rightarrow \mathbb{R}$ is such that $L(x, \cdot )$ is strongly convex, uniformly in $x$. Additionally, suppose Assumption \ref{asmp:cost} holds. Then, consider the function,
    \[\alpha(\omega,p) := \arg\min_{\alpha \in \mathbb{R}^m} \left[\langle p, f(\omega,\alpha) \rangle + L(\omega,\alpha)\right]\]
    Then Assumption \ref{asmp:pmp-control} holds true.
    \label{prop:relaass}
\end{proposition}

\begin{proof}

Since $L(\omega,\cdot)$ is uniformly strongly convex and $
a \mapsto \sum_{i=1}^m a_i \langle p,f_i(\omega)\rangle $
is affine, the function
$
a \mapsto \Phi(\omega,p,a): =\langle p,f(\omega,a)\rangle + L(\omega,a)
$
is strongly convex for every $(\omega,p)$. Moreover, because the vector fields $f_i$ are globally bounded and Assumption \ref{asmp:cost} gives quadratic growth in the control variable, $\Phi(\omega,p,\cdot)$ is coercive. Hence $\Phi(\omega,p,\cdot)$ admits a unique minimizer, which we denote by $\alpha(\omega,p)$.

We show local Lipschitz continuity using a standard implicit function theorem argument.
Define
$G(\omega,p,a):=\nabla_a \Phi(\omega,p,a).$
Since $f$ is affine in $a$,
\[
G(\omega,p,a)
=
\bigl(\langle p,f_1(\omega)\rangle,\dots,\langle p,f_m(\omega)\rangle\bigr)
+
\nabla_a L(\omega,a).
\]
The map $G$ is $C^1$ because the vector fields are smooth and $L\in C^2$.

By construction, $\alpha(\omega,p)$ satisfies the first-order optimality condition
\[
G(\omega,p,\alpha(\omega,p))=0.
\]
Moreover,
\[
D_aG(\omega,p,a)=D^2_{aa}L(\omega,a).
\]
Since $L(\omega,\cdot)$ is strongly convex uniformly in $\omega$, there exists $\lambda>0$ such that
\[
\xi^\top D^2_{aa}L(\omega,a)\,\xi \ge \lambda |\xi|^2
\qquad
\text{for all }(\omega,a,\xi)\in\mathbb R^d\times\mathbb R^m\times\mathbb R^m.
\]
Hence $D_aG(\omega,p,a)$ is invertible for every $(\omega,p,a)$.
Applying the implicit function theorem at each point $
(\omega,p,\alpha(\omega,p))$, we conclude that $(\omega,p)\mapsto \alpha(\omega,p)$ is $C^1$ locally.

Next, we derive the linear growth estimate.
By minimality of $\alpha(\omega,p)$,
\[
\Phi(\omega,p,\alpha(\omega,p))\le \Phi(\omega,p,0).
\]
This becomes equivalent to,
\[
L(\omega,\alpha(\omega,p))
+
\sum_{i=1}^m \alpha_i(\omega,p)\,\langle p,f_i(\omega)\rangle
\le
L(\omega,0).
\]
Using again Assumption \ref{asmp:cost} and boundedness of the $f_i$,
\[
\theta |\alpha(\omega,p)|^2
\le
L(\omega,0)+ C_f\sqrt m\,|p|\,|\alpha(\omega,p)|.
\]
with $C_f$ the maximum of the upper bounds of $\sup_{x \in \mathbb{R}^d} |f_i(x)|$.
We will now bound $L(\omega,0)$. Since Assumption \ref{asmp:cost} gives
\[
|\nabla_xL(x,u)|\le C(1+|x|)
\qquad \text{for all }(x,u) \in \mathbb{R}^d \times \mathbb{R}^m,
\]
we may integrate along the segment $s\mapsto s\omega$ to obtain
\[
|L(\omega,0)-L(0,0)|
\le
\int_0^1 |\nabla_xL(s\omega,0)\cdot \omega|\,ds
\le
C|\omega|\int_0^1 (1+s|\omega|)\,ds
\le
C(1+|\omega|^2).
\]
Therefore
\[
L(\omega,0)\le C(1+|\omega|^2).
\]
Substituting this into the previous estimate yields
\[
\theta |\alpha(\omega,p)|^2
\le
C(1+|\omega|^2)+C|p|\,|\alpha(\omega,p)|.
\]
Applying Young's inequality to the last term,
\[
C|p|\,|\alpha(\omega,p)|
\le
\frac{\theta}{2}|\alpha(\omega,p)|^2 + C|p|^2,
\]
and hence
\[
|\alpha(\omega,p)|^2 \le C(1+|\omega|^2+|p|^2).
\]
Taking square roots,
\[
|\alpha(\omega,p)|\le M(1+|\omega|+|p|),
\]
for some constant $M>0$.

This proves Assumption \ref{asmp:pmp-control} holds.
\end{proof}

\appendix

\bibliographystyle{plain}
\bibliography{ref}

\end{document}